\documentclass[a4paper]{amsart}

\usepackage[english]{babel}
\usepackage[utf8]{inputenc}
\usepackage{amssymb,amsmath,amsthm,amsfonts}
\usepackage{graphicx}
\usepackage[usenames,dvipsnames]{xcolor}
\usepackage{cancel}
\usepackage[normalem]{ulem}
\usepackage{dsfont}
\usepackage{enumitem}
\usepackage{microtype}
\usepackage[showonlyrefs]{mathtools}
\mathtoolsset{showonlyrefs=true}

\usepackage{tikz}
\usetikzlibrary{matrix, calc, positioning, decorations.pathreplacing, arrows.meta}
\usepackage[colorlinks]{hyperref}

\newcommand{\N}{\mathbb{N}}

\newcommand{\R}{\mathbb{R}}

\newcommand{\eps}{\varepsilon}

\newcommand{\Om}{\Omega}

\newcommand{\mean}[1]{\,-\hskip-1.08em\int_{#1}} 

\newcommand{\Ecal}{{\mathcal{E}}}

\newcommand{\resmeas}{\mathbin{\vrule height 1.6ex depth 0pt width 0.13ex\vrule height 0.13ex depth 0pt width 1.3ex}} 

\newcommand{\supp}{\text{supp}}
\renewcommand{\H}{H^1_0}
\renewcommand{\phi}{\varphi}

\def\Omt{\widetilde{\Om}}
\def\Omh{\widehat{\Om}}
\def\mh{\widehat{m}}

\def\N{\mathbb{N}}

\def\ut{\widetilde{u}}
\def\uh{\widehat{u}}

\def\de{\delta}
\def\hc{\mathcal{H}}

\newtheorem{proposition}{Proposition}[section]
\newtheorem{theorem}[proposition]{Theorem}

\newtheorem{lemma}[proposition]{Lemma}
\theoremstyle{definition}
\newtheorem{definition}[proposition]{Definition}
\newtheorem{remark}[proposition]{Remark}

\title[Hartree--Ohta-Kawasaki in the small mass regime]{Existence and Regularity in the Small-Mass Regime for a Hartree--Ohta-Kawasaki Shape Optimization Problem}

\author{Dario Mazzoleni}
\address{Dipartimento di Matematica F. Casorati\\
		Universit\`a di Pavia\\
		Via Ferrata 5, 27100 Pavia, Italy}
\email{dario.mazzoleni@unipv.it}

\author{Riccardo Moraschi}
\address{Dipartimento di Matematica F. Casorati\\
		Universit\`a di Pavia\\
		Via Ferrata 5, 27100 Pavia, Italy}
\email{r.moraschi1@campus.unimib.it}

\author{Berardo Ruffini}
\address{Dipartimento di Matematica, Universit\`a di Bologna, Piazza di Porta, S.Donato
5, 40126, Bologna-Italy}
\email{berardo.ruffini@unibo.it}

\begin{document}

\begin{abstract}
We consider a shape optimization problem for a hybrid energy combining local confinement and nonlocal Coulomb repulsion. Specifically, for any open set $\Omega \subseteq \R^3$ of prescribed volume, we consider the ground state energy of an $L^2$-normalized function supported in $\Omega$, defined as a linear combination of its homogeneous $\dot{H}^1$ and $\dot{H}^{-1}$ seminorms. We show that in the small mass regime, volume-constrained minimizers of this geometric functional exist and are $C^{2,\alpha}$ perturbations of a ball. The proof relies on a combination of surgery techniques, $\Gamma$-convergence, elliptic PDE theory, and one-phase free boundary regularity.

A key novelty of this paper lies in the treatment of the Coulombic repulsive term: unlike standard competitive models, the lack of (a priori) sign constraints on the optimal functions forces the nonlocal term to exhibit two natures: it acts both as a \emph{scattering} and an \emph{homogenizing} force. 
\end{abstract}

	\maketitle

	\tableofcontents

    \smallskip


\smallskip

	
\section{Introduction}\label{intro}
\subsection{Foreword}
Consider the energy functional
\[
\widetilde E(\Omega)=\min \left\{  [u]^2_{\dot{H}^1}+[u]^2_{\dot{H}^{-1}}  \,:\, \|u\|_{L^2(\Omega)}=1,\,\, u\in H^{1}_0(\Om) \right\}.
\]
We aim to study the minimizers of $\widetilde E$ among open subsets of $\R^3$ with mass constraint $|\Omega|=m$. 
Since we are interested in an optimal design problem where the set $\Omega$ is free to move on $\R^3$, it is convenient to consider the seminorm as defined on the whole space, that is, we extend  $u\in H^1_0(\Omega)$ to zero outside $\Omega$ and  let, up to a multiplicative constant, \[
[u]^2_{\dot{H}^1}=\int_\Omega |\nabla u|^2\,dx,
\qquad
 \text{ and } 
\qquad 
[u]^2_{\dot{H}^{-1}}=\int_{\R^3}\int_{\R^3}\frac{u(x)u(y)}{|x-y|}\,dxdy=\int_{\Omega}\int_{\Omega}\frac{u(x)u(y)}{|x-y|}\,dxdy.
\]
Then, we define
\begin{equation}\label{eq: EqOmegau}
	E_q(u,\Omega):= \int_\Om |\nabla u(x)|^2\, dx+\frac{q}{2}\int_\Om \int_\Om \frac{u(x)u(y)}{|x-y|}\, dxdy
\end{equation}
and we consider the energy functional
\begin{equation}\label{eq: EqOmega}
	E_q(\Om):= \min \left\{E_q(u,\Omega): u\in \H(\Om),\; \int_\Om u^2(x)\, dx=1 \right\}.
\end{equation}
Up to take $q=2\Big(\frac{m}{|B_1|}\Big)^{4/3}>0$, thanks to a scaling argument, the original problem is equivalent to 
\begin{equation}\label{eq: inf EqOmega}
	\min \left\{ E_q(\Om): \Om \subseteq \R^3,\; \text{open, }\; |\Om|=|B_1| \right\}.
\end{equation}
We are interested in this paper in the small mass regime $m\ll1$, that is, $q\ll1$.
\vspace{.3cm}

When $q=0$, $E_q(\Omega)$ is the first eigenvalue of the Dirichlet Laplacian. The well-known Faber-Krahn inequality assures that it is uniquely minimized by balls, up to null capacity, among sets of fixed volume. The case $q>0$ is less clear, as the second addend in the energy may act as a scattering term. Note in particular that among \emph{positive} functions, the second term in the energy is maximized by balls, by the Riesz rearrangement inequality, see \cite{LiebLoss}; instead,  the minimization of the second term alone without sign constraints leads to an  ill posed problem, as it is shown in Appendix~\ref{appendix}. In this case we expect the positive and negative parts of the optimal functions to homogeneize in order to favor cancellations, see also Section~\ref{sec:motivation} below. These formal observations reveal a sharp contrast between the two addends in the energy. 

Nevertheless, also in view of the recent results on the stability of isoperimetric inequalities \cite{FuMaPra2008,FiMaPra2010,CicLeo2012},  and their spectral counterparts \cite{BrascoDePhilippisVelichkov,DePhilippisMariniMukoseeva}, and since the ball is a critical point for the nonlocal term, it can be expected that for $q$ sufficiently small the ball is a stable minimizer. 

In  this paper we  show that in the small mass regime well posedness is restored. Namely we show that minimizers exist and that they do converge in $C^{2,\alpha}-$norm to a ball. Specifically, this is the main result of the paper.

\begin{theorem}\label{thm:main}
There exist universal\footnote{In this paper we call \emph{universal constant} a constant possibly depending only on the dimension of the ambient space, which, in our case, is $d=3$.}  constants $q^*>0$ and
          $\alpha\in (0,1)$ such that, for all $0 < q \leq q^*$,
          there exists an optimal set for
          problem~\eqref{eq: inf EqOmega}. Furthermore, every optimal
          set $\Omega_q$ is $C^{2,\alpha}$-nearly spherical, namely there is a function
          $\varphi_q\colon \partial B_1 \to \R$ of class
            $C^{2,\alpha}$ such that
          $\|\varphi_q\|_{C^{2,\alpha}}$ vanishes as $q\to 0$ and
		\[ \partial\Omega_q=\Big\{( 1 +\varphi_q(x)) x : x\in \partial B_1
		\Big\}.
		\]
\end{theorem}

\subsection{Motivation and background}\label{sec:motivation}

The functional we study can be viewed as an analytical hybrid model combining a spectral confinement energy with a long-range Coulombic self-interaction. From this perspective, the optimization over the shape 
$\Omega$ describes the search for geometries minimizing the energy of a confined state under weak nonlocal repulsion. 

\vspace{.2 cm}
The competition between a local regularizing term and a nonlocal long-range interaction is a standard feature in pattern formation and microphase separation, as it arises in actual models such as optimal design of charged quantum devices \cite{MazzoleniMuratovRuffini} or Ohta-Kawasaki models, see \cite{ChoksiPeletier,ACO09,DR26} and references therein.
  Let us make a comparison between our case and the well-known Ohta-Kawasaki energy for diblock copolymers. Analytically, our problem acts as a diffuse spectral counterpart to the sharp-interface model studied in the seminal work of Alberti, Choksi, and Otto \cite{ACO09}. In the classical Ohta-Kawasaki setting, the competition between a local perimeter term  and a nonlocal Coulombic repulsion under an $L^1$ mass constraint leads to pattern formation for large masses. In our functional, the local regularizing role of the perimeter is carried out  by the Dirichlet energy.

The lack of an a priori sign constraint on the admissible functions $u \in H^1_0(\Omega)$ implies a deep difference in how the system reacts to the nonlocal repulsion. In the regime  $q\gg1$ (corresponding to the large mass regime in \cite{ACO09}), the Coulombic term dominates so that, to minimize this repulsive term, the energy does not just separate the domain into periodic patterns as in \cite{ACO09}, but the optimal function $u$ tends to oscillate between positive and negative values to maximize cancellations in the double integral. 
In this regime, the nonlocal term acts as a strongly homogenizing force. We show, at least formally, this phenomenon in Appendix A.

Conversely, for $q\ll1$, the spectral confinement dominates the behavior of the energy. The optimal function is a posteriori forced to be strictly positive, as the energy is close, for $q\ll1$, to the first Dirichlet eigenvalue. At this point, as the positivity is established, the Coulombic term changes completely its nature, and acts as  a (more standard) scattering force.

\vspace{.2cm}

One expects that in this weak-coupling regime ($q\ll1$), the spectral term dominates, making the ball optimal. 
 A common strategy to prove such a stability for local/nonlocal interactions was first exploited in \cite{KM1,KM2}, see also \cite{CiCSpa}, following the technique -referred to as \emph{selection principle}- developed in \cite{CicLeo2012} in order to prove the quantitative isoperimetric inequality.
Such a strategy roughly consists of dividing the proof into two steps: show rigidity (i.e. that the ball is the unique minimizer) in a suitably chosen class of regular competitors, the nearly spherical sets, and prove independently that a minimizer (or any minimizing sequence) lies in such a class.

In this paper we tackle the latter part of this plan: Theorem~\ref{thm:main} shows that all minimizers are a $C^{2,\alpha}$-small perturbation of a ball. In fact, a fundamental byproduct of our proof is that in the small$-q$ regime, the optimal functions in $\eqref{eq: EqOmegau}$ are positive so that the ball is in fact a critical point of the energy, thanks to P\'olya-Szeg\"o inequality and Riesz rearrangement inequality. This, together with preliminary computations, leads us to formulate the following conjecture.

\vspace{.4cm}

\begin{center}
\emph{
{\bf Conjecture:}
There exists a universal constant $q_0>0$ such that for $ q<q_0$ balls are rigid minimizers for Problem \eqref{eq: inf EqOmega}.}
\end{center}

\vspace{.4cm}

The above conjecture deserves some comments. In order to show stability of the ball among nearly spherical sets, there are two techniques available that we are aware of. Either to perform a second order shape derivative of the full energy and to show that this is positive in a suitable neighborhood of a ball, or to show that the Euler-Lagrange equation for a minimizer produces an overdetermined problem which can be satisfied only by balls. The first idea has been successfully exploited in several recent results, but only in the case in which the two addends in the functional are independent shape functionals, i.e. when the full energy is the sum of an aggregating shape functional term $\mathcal A$  and a repulsive one $\mathcal{B}$, i.e. of the form
\[
\Omega\mapsto \mathcal A(\Omega)+\mathcal B(\Om). 
\]
See, for instance, \cite{KM2,GNR2025,GNR1,BrascoDePhilippisVelichkov,MazzoleniRuffini}. However, the derivation of a second order shape derivative for nonlocal functionals is a challenging open problem.
The second strategy appears hard, too. Indeed, despite the huge literature developed from the pioneering works by Serrin and Weinberger \cite{Serrin, Weinberger}, there is not a clear understanding of how to approach overdetermined problems with competing non-local terms. Even the new free boundary approaches recently developed for example in~\cite{FigalliZhang} do not seem to work in our setting.

\vspace{.3cm}

\subsection{Detailed strategy}

In order to make the reading more clear, we offer here below the detailed outline of the strategy of the proof of Theorem \ref{thm:main}. As mentioned just above, this  is inspired by that developed in \cite{MazzoleniMuratovRuffini,MazzoleniRuffini} with a major technical difference due to the fact that an optimal function $u$ is not a priori positive.  
Neglecting the spectral confinement, i.e. formally picking $q=+\infty$, a minimizing sequence driven only by the nonlocal term would seek to maximize cancellations, leading to a homogenization problem. This sign-changing phenomenon completely prevents the direct application of standard shape optimization techniques or classical overdetermined problem strategies. To bridge this gap, we develop a  strategy mixing $\Gamma$-convergence and quasi-minimizer regularity \`a-la Giaquinta-Giusti.

Hence, this marks a substantial difference between our stability result and \emph{all} optimal design problems with competitive terms that we are aware of in the literature. In order to deal with this issue, we need to strongly use the fact that $\Omega$ solves a minimization problem, and often merge standard variational techniques with PDE ones. 

The key points of our strategy are the following:
\begin{itemize}
\item[1.] First, we prove some basic properties of optimal functions for $E_q(\Omega)$, in particular that they are uniformly bounded when $q\in(0,1]$, and we show that we can get rid of the $L^2$ integral constraint by adding a suitable Lagrange multiplier.
\item[2.] Then, we prove a surgery result that allows to work in an equibounded setting, that is, where the competing shapes have equibounded diameter. This result is somehow independent from all of the other parts of the proofs. This part is quite long and technical but essentially follows the surgery technique developed in \cite{MazzoleniPratelli} with some nontrivial technical modifications.
\item[3.] After this, we follow a nowadays classical approach which was first developed by Aguilera, Alt and Caffarelli in \cite{AguileraAltCaffarelli} to get rid of the measure constraint.
\item[4.] At this point we work with a complete  unconstrained functional. Now, with a simple yet delicate argument mixing up $\Gamma$-convergence tools and a notion of quasi-minimizers for the Dirichlet energy \`a-la Giaquinta-Giusti, we are able to prove that the optimal functions $u_q$ for optimal shapes $\Omega_q$ can be chosen to be non-negative, if $q$ is small enough.

\end{itemize}

\vspace{.3cm}

The last point is  a necessary bridge. 
The remaining part of the scheme of the proof  is as follows. We prove regularity of the boundary by means of regularity tools developed in the one-phase free boundary theory. 
Also the understanding of this second  part  may be eased by a detailed strategy.  

\vspace{.3cm}

\begin{itemize}
\item[5.] We begin with some regularity result obtained, as mentioned above, via  tools of the regularity for one-phase free boundary problems: existence of minimizers, finiteness of their perimeter,  their non-degeneracy, and Lipschitz regularity of the optimal functions. 
\item[6.] Now, via a scaling argument we are ready to show the equivalence between the original problem with the measure constraint and this unconstrained problem (for $q$ small).
\item[7.] At this point we perform a first shape variation argument around points of the reduced boundary (which covers almost all of the boundary, as our sets are of finite perimeter after point 5.) to show an optimality condition at the free boundary. This may be read as a weak Euler-Lagrange equation for the energy. 
\item[8.] Finally we  apply the improvement of flatness machinery by Alt and Caffarelli to prove the main result of the paper \cite{AltCaffarelli}. The regularity, together with a simple localization argument which is a consequence of the quantitative stability of the Faber-Krahn inequality \cite{BrascoDePhilippisVelichkov}, will be enough to  prove  Theorem \ref{thm:main}.
\end{itemize}

Figure~\ref{fig:schema_matematico} represents the tree of the steps of the proof detailed here above.


\definecolor{boxbg}{HTML}{F8F7FF}
 \definecolor{boxborder}{HTML}{7A6BA1}



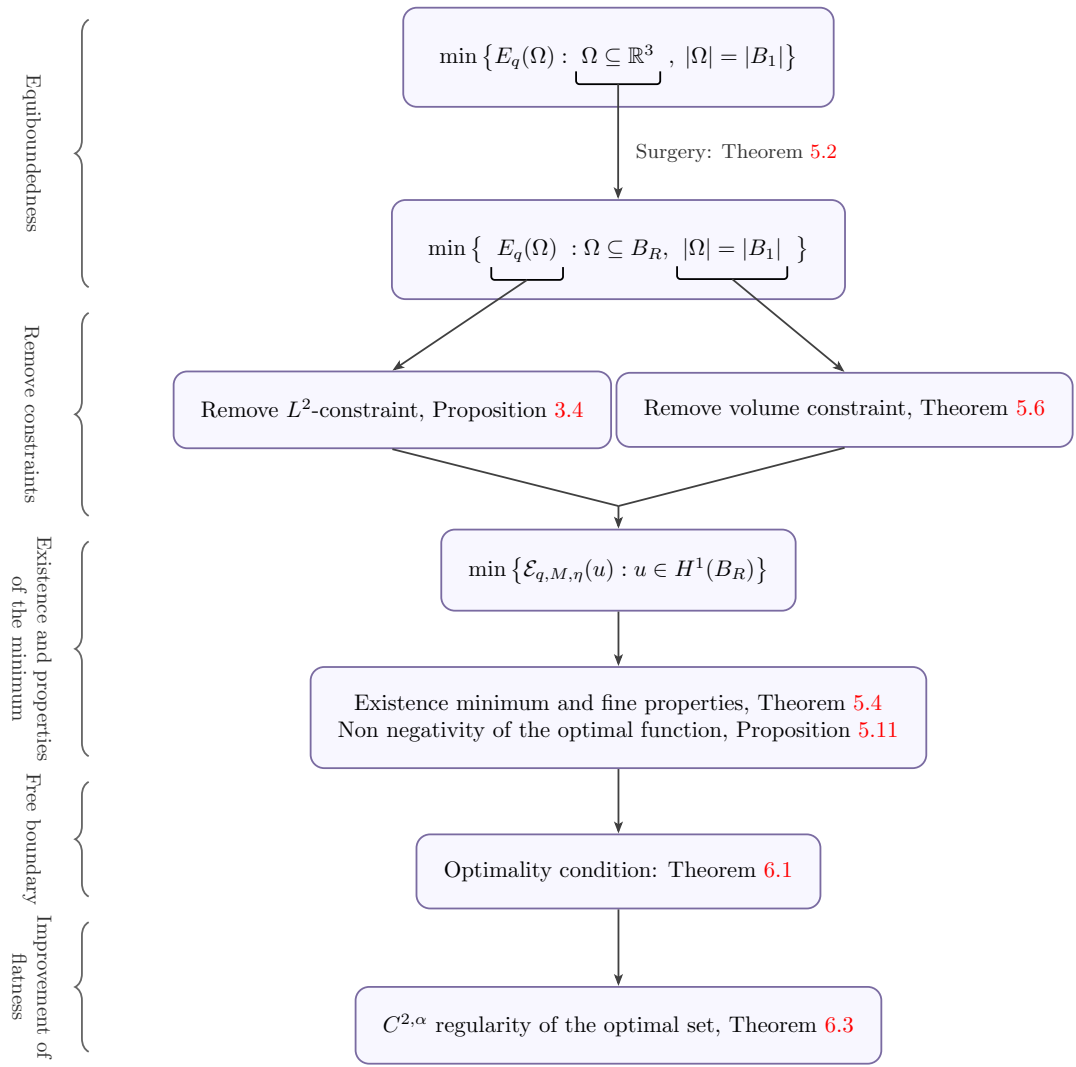
\begin{figure}[htbp]
    \centering
    \scalebox{0.85}{
    \begin{tikzpicture}[
        >={Stealth[scale=0.8, inset=1.5pt]},
        txt/.style={align=center},
        pbox/.style={
            draw=boxborder, thick, rounded corners=6pt,
            fill=boxbg, inner sep=12pt, align=center
        },
        pmatrix/.style={
            matrix of math nodes,
            ampersand replacement=\&,
            draw=boxborder, thick, rounded corners=6pt,
            fill=boxbg, inner sep=14pt, column sep=2pt,
            nodes={inner sep=2pt}
        },
        arr/.style={->, thick, draw=black!75},
        line/.style={thick, draw=black!75},
        sidebrace/.style={
            decorate, decoration={brace, amplitude=6pt},
            thick, draw=black!60
        },
        sidetxt/.style={
            midway, xshift=-25pt, rotate=-90, align=center, font=\small\color{black!80}
        }
    ]

    \matrix (m1) [pmatrix] at (0, 0) {
        \min \big\{ E_q(\Omega) : \& \Omega \subseteq \mathbb{R}^3 \& ,\; \vert\Omega\vert = \vert B_1\vert \big\} \\
    };
    \draw[thick, draw=black, color=black, rounded corners=1.5pt] (m1-1-2.south west) -- ++(0,-0.2) coordinate (b1) -- (m1-1-2.south east |- b1) -- (m1-1-2.south east);
    \coordinate (b1_mid) at ($(b1)!0.5!(m1-1-2.south east |- b1)$);

    \matrix (m2) [pmatrix] at (0, -3.0) {
        \min \big\{ \& E_q(\Omega) \& : \Omega \subseteq B_R, \& \vert\Omega\vert = \vert B_1\vert \& \big\} \\
    };
    \draw[arr] (b1_mid) -- (b1_mid |- m2.north) 
    node[pos=0.6, right=4pt, font=\small, color=black!80] 
    {Surgery: Theorem \ref{thm: equivalence unbounded and bounded problem}};

    \draw[thick, draw=black, color=black, rounded corners=1.5pt] (m2-1-2.south west) -- ++(0,-0.2) coordinate (b2a) -- (m2-1-2.south east |- b2a) -- (m2-1-2.south east);
    \coordinate (b2a_mid) at ($(b2a)!0.5!(m2-1-2.south east |- b2a)$);

    \draw[thick, draw=black, color=black, rounded corners=1.5pt] (m2-1-4.south west) -- ++(0,-0.2) coordinate (b2b) -- (m2-1-4.south east |- b2b) -- (m2-1-4.south east);
    \coordinate (b2b_mid) at ($(b2b)!0.5!(m2-1-4.south east |- b2b)$);

    \node[pbox] (rem_l2) at (-3.5, -5.5) {Remove $L^2$-constraint, Proposition \ref{prop:equivalence E_q and E_q,M}};
    \node[pbox] (rem_vol) at (3.5, -5.5) {Remove volume constraint, Theorem \ref{thm:noconstraint}};
    \draw[arr] (b2a_mid) -- (rem_l2.north);
    \draw[arr] (b2b_mid) -- (rem_vol.north);
    
    \coordinate (converge) at (0, -7.0); 
    \draw[line] (rem_l2.south) -- (converge);
    \draw[line] (rem_vol.south) -- (converge);

    \node[pbox] (m3) at (0, -8.0) {$\min \big\{ \mathcal{E}_{q,M,\eta}(u) : u \in H^1(B_R) \big\}$};
    \draw[arr] (converge) -- (m3.north);

    \node[pbox, txt] (txtbox1) at (0, -10.3) {Existence minimum and fine properties, Theorem \ref{thm:existenceEq,M,eta} \\ Non negativity of the optimal function, Proposition \ref{prop: positivity of the minimizers}};
    \draw[arr] (m3.south) -- (txtbox1.north);

    \node[pbox, txt] (txtbox2) at (0, -12.7) {
    Optimality condition: Theorem \ref{thm:agalcathm2}};
    \draw[arr] (txtbox1.south) -- (txtbox2.north);

    \node[pbox, txt] (txtbox3) at (0, -15.1) {$C^{2,\alpha}$ regularity of the optimal set, Theorem \ref{thm:bdvthm4.18}};
    \draw[arr] (txtbox2.south) -- (txtbox3.north);

    \coordinate (L) at (-8.2, 0); 
    \draw[sidebrace] ($(m2.south -| L) + (0, 0.2)$) -- ($(m1.north -| L) - (0, 0.2)$) node[sidetxt] {Equiboundedness};
    \draw[sidebrace] ($(m3.north -| L) + (0, 0.2)$) -- ($(m2.south -| L) - (0, 0.2)$) node[sidetxt] {Remove constraints};
    \draw[sidebrace] ($(txtbox1.south -| L) + (0, 0.2)$) -- ($(m3.north -| L) - (0, 0.2)$) node[sidetxt] {Existence and properties\\of the minimum};
    \draw[sidebrace] ($(txtbox2.south -| L) + (0, 0.2)$) -- ($(txtbox1.south -| L) - (0, 0.2)$) node[sidetxt] {Free boundary};
    \draw[sidebrace] ($(txtbox3.south -| L) + (0, 0.2)$) -- ($(txtbox2.south -| L) - (0, 0.2)$) node[sidetxt] {Improvement of\\flatness};

    \end{tikzpicture} }
    \caption{Scheme of the proof}
    \label{fig:schema_matematico}
\end{figure}

We conclude this introduction with some  observations.
\begin{remark}
Another way of looking at our problem, which was used, in a different setting, in~\cite{BenguriaPereira}, is to see the energy $E_q(\Omega)$ as an eigenvalue of a competitive mixture of $-\Delta$ and $(-\Delta)^{-1}$. The study of mixtures of operators, one local and the other nonlocal, has been the subject of several works in the last few years, mostly in homogeneous settings. See for instance \cite{BDVV26} and references therein. We just remark that our proof does not apply to nonintegrable kernels in the nonlocal term, as the $s$-Laplacian would be.
\end{remark}

\begin{remark}
This paper is settled in the physical three dimensional case. From the purely mathematical point of view, one could of course try to generalize the
problem to other dimensions, or one could consider more general Riesz kernels, a general power $p$
for the gradient term and general powers $q$ and $r$
in the non-local term
and the constraint, respectively. Such generalizations often appear in the literature and
lead to a great variety of bifurcations as the dimension and other parameters are varied. While we
believe that our results may be generalized to higher dimensions and more general Riesz kernels (with suitably chosen exponents), such  generalization would, highly complicate all
the notations of our paper obscuring its main points. Therefore, instead of carrying out
such an extension, we prefer to limit ourselves to the only physically relevant case of $N=3$ and the
Newtonian potential.
\end{remark}

{\bf Structure of the paper.} 
After this Introduction, Section~\ref{sec:prelim} is devoted to setting notations and preliminary results needed in the rest of the paper.
Section~\ref{sec:Eq} deals with the study of the minimization of $E_q(u,\Omega)$ with respect to $u$ only and to the reformulation without the $L^2$ constraint.
Then in Section~\ref{sec:surgery} we prove the surgery result mentioned  in point $2.$ above.
Section~\ref{sec:Auxiliary} is devoted to introducing the unconstrained functional (point $3.$), to prove that optimal functions are positive for $q$ small (point $4.$) and to obtain first mild regularity properties (point $5.$).
Finally, in Section~\ref{sec:improvementofflatness} an optimality condition at the free boundary is proved and this allows to conclude by the classical improvement of flatness for one-phase problems (points $6.-8.$).

\section{Preliminaries}\label{sec:prelim}
Throughout the paper, we adopt the following notations: $\Omega \subset \R^3$ is   a  quasi-open  set\footnote{See the end of this section for the precise notion of quasi-open sets} of finite measure and
$u\in H^1_0(\Omega)$ denoting with $*$ the usual
convolution, we define
\begin{equation}\label{eq:defvu}
	v_u(x):=\left(u * \frac{1}{|\cdot|}\right) (x) =\int_\Omega\frac{u(y)}{|x-y|}\,dy,
\end{equation}
with $v_u \in W^{2,6}_\mathrm{loc}(\R^3)$ since $u\in H^1_0(\R^3)\hookrightarrow W^{1,6}(\R^3)$ and by \cite[Theorem 9.9]{GilbargTrudinger}. We will use several times the following Hardy-Sobolev inequality  
\begin{equation}\label{eq:Hardy-Sobolev}
\int_{\R^3}\frac{|\phi(x)|}{|x-y|}dx\leq C\int_{\R^3}|\nabla \phi(x)|dx \quad \forall\phi \in C_c^\infty(\R^3),
\end{equation}
for a universal constant $C>0$, see for instance \cite[Corollary 2 of Section 2.1.7]{Mazya}.  
The inequality holds also for $u\in \H(\Om)$ by approximation and Fatou lemma. We introduce, for $\phi,\psi:\Omega \to \R$, the Coulomb energy 
\[ D(\phi,\psi):=\int_\Omega\int_\Omega \frac{\phi(x)\psi(y)}{|x-y|}dxdy.\]
Note that it is a well defined bilinear form as soon as $ D(|\phi|,|\psi|)<+\infty$. 
Let us state some basic properties of the Coulomb interaction. 
\begin{lemma}\label{le:positivityD(u,u)}
    It holds that $D(u,u)\geq0$ for all $u\in \H(\Om)$. Moreover it holds that $D(u,u)>0$ if $u\in \H(\Om)\setminus\left\{0\right\}$ and the map $u\to D(u,u)$ is strictly convex in $\H$.
\end{lemma}
For the proof of the above Lemma, we refer to \cite[Theorem 9.8]{LiebLoss}. 

\begin{proposition}\label{prop:continuity of D(u,u)}
Let $\Omega$ be a quasi-open set of finite measure. Then the functional $D:L^2(\Omega)\to\R$ defined by
	$	D(u)=  D(u,u)$	
	is continuous in $L^2(\Omega)$.
	\begin{proof}
		Let $u\in L^2(\Om)$ and $(u_n)_n$ be a sequence such that $u_n\to u$ in $L^2(\Omega)$. Then, up to subsequence, there exists $g\in L^2(\Omega)$ such that $|u_n(x)|\leq g(x)$ and $u_n(x)\to u(x)$ almost everywhere in $x\in \Omega$. Since \[ \frac{|u_n(x)u_n(y)|}{|x-y|}\leq \frac{|g(x)g(y)|}{|x-y|} \quad \text{ for a.e. } 
        (x,y)\in \Omega\times \Omega,\]
		then if  $\frac{|g(x)g(y)|}{|x-y|}\in L^1(\Omega\times\Omega)$, the claim holds by Dominated Convergence Theorem. Let $v_g(x):=\int_{\Omega}\frac{g(y)}{|x-y|}dy$ for $x\in\Omega$, then $v_g\in W^{2,6}(\Omega)$  by \cite[Theorem 9.9]{GilbargTrudinger} and so 
		 \begin{equation*}
		 	\int_\Omega\int_\Omega \frac{|g(x)g(y)|}{|x-y|}dxdy=\int_\Omega v_g(x)g(x)dx\leq ||v_g||_{L^2(\Omega)}||g||_{L^2(\Omega)}<\infty. \qedhere
		 \end{equation*}

	\end{proof}
\end{proposition}

\subsection{Quasi-open sets}

In this section, we briefly recall some standard definitions and properties concerning quasi-open sets, which define  the natural setting on which our shape minimization problem is well posed. We refer to  \cite{HenrotPierre2018} for further details on the subject.

\begin{definition}
A measurable set $\Omega \subset \mathbb{R}^3$ is called \textit{quasi-open} if, for every $\varepsilon > 0$, there exists a compact set $K_{\varepsilon}$ whose Newtonian capacity satisfies ${\rm cap}(K_{\varepsilon}) < \varepsilon$ such that  $\Omega \setminus K_{\varepsilon}$ is an open set. Here ${\rm cap}(\cdot)$ stands for the standard Sobolev capacity.
\end{definition}

Along the same lines, a function $u: \Omega \to \mathbb{R}$ is defined as \textit{quasi-continuous} if for every $\varepsilon > 0$ one can find a compact set $K_{\varepsilon}$ with ${\rm cap}(K_{\varepsilon}) < \varepsilon$ such that the restriction of $u$ to $\Omega \setminus K_{\varepsilon}$ is continuous. A property is said to hold \textit{quasi-everywhere} (q.e.) if the set where it fails to hold has null capacity.

It is a classical result that any function $u \in H^1(\Omega)$ admits a quasi-continuous representative $\tilde{u}:\Omega\to\R$. Furthermore, any two quasi-continuous representatives of $u$ coincide quasi-everywhere. Consequently, it is customary to identify every Sobolev function  with its quasi-continuous representative. In this framework, a quasi-open set can simply be understood as a superlevel set of a function $u \in H^1(\R^3)$. Notice that if a function $u\in H^1(\R^3)$ is sign-changing, clearly $\{u\not=0\}$ is a quasi-open set.

Finally, for a quasi-open set $\Omega \subset \mathbb{R}^3$, the associated Sobolev space $H_0^1(\Omega)$ is naturally defined as:
\begin{equation}
    H_0^1(\Omega) = \{u \in H^1(\mathbb{R}^3) : u = 0 \text{ quasi-everywhere in } \mathbb{R}^3 \setminus \Omega\}.
\end{equation}
 We stress that this definition coincides with the classical one whenever $\Omega$ is a standard open set.

\subsection{Fraenkel asymmetry and quantitative Faber-Krahn inequality}

We conclude  this preliminary section by recalling the sharp quantitative version of the Faber-Krahn inequality. In order to do that we first remind the notion of \textit{Fraenkel asymmetry}: for a measurable set $E \subset \mathbb{R}^3$ with finite measure we define
\begin{equation}\label{eq: fraenkel}
	\mathcal{A}(E) = \inf_{x\in\mathbb{R}^3} \frac{|E \Delta(B + x)|}{|E|},
\end{equation}
where $B$ denotes the ball of measure $|E|$ centered at the origin. We also recall that we denote by $\lambda_0(\Omega)$ the first eigenvalue of the Dirichlet Laplacian of a quasi-open set $\Omega \subset \mathbb{R}^3$.

\begin{theorem}[\cite{BrascoDePhilippisVelichkov}] \label{thm:quantitativefk}
	There exists a universal positive constant $\widehat{\sigma} > 0$ such that for all quasi-open sets $\Omega \subset \mathbb{R}^3$ with finite measure we have
	\[
		|\Omega|^{2/3}\lambda_0(\Omega) - |B_1|^{2/3}\lambda_0(B_1) \geq \widehat{\sigma}\mathcal{A}(\Omega)^2.
	\]

\end{theorem}

\section{Minimization of $E_q(u,\Omega)$ in $u$}\label{sec:Eq}
We are now in a position to address the study of the energy functional for  a fixed domain $\Omega$. Relying on the nonnegativity and continuity properties established in the previous section, the first natural result guarantees the existence of a minimizing function.
\begin{lemma}\label{le:exis u min}
	Let $\Omega\subset \R^3$ be a quasi-open set of finite measure and
	let $q\geq0$.  Then 
	\begin{equation*}
		E_q(\Omega)= \inf \left\{E_q(v,\Omega) : v\in \H(\Omega),\:	\int_\Omega v^2\, dx=1\right\}. 
	\end{equation*}
	admits a minimizer. 
\end{lemma}

\begin{proof}
	Let $(u_n)_n$  be a minimizing sequence for the energy. By Lemma \ref{le:positivityD(u,u)} all the terms in the definition of $E_q(u,\Omega)$ are nonnegative, thus we infer that $(u_n)_n$ is bounded in $H^1_0(\Omega)$.
	Then up to passing to a subsequence, $(u_n)_n$ converges weakly in
	$H^1_0(\Omega)$, strongly in $L^2(\Omega)$ and pointwise a.e. to some function
	$u\in H^1_0(\Omega)$. Observe that in the case of finite measure unbounded set, the Sobolev's embedding holds by \cite{BucurButtazzo}. In particular the $L^2-$convergence implies
	that $\|u\|_{L^2(\Omega)}=1$. By lower semicontinuity with respect to
	the weak convergence, we have that
	\[
	\int_\Omega |\nabla u|^2\,dx\le \liminf_{n\to+\infty} \int_\Omega |\nabla u_n|^2\,dx
	\]
	and, by Proposition \ref{prop:continuity of D(u,u)}
	\[
	D(u,u)=\iint_{\Omega\times\Omega}\frac{u(x)u(y)}{|x-y|}\,dxdy= \lim_{n\to+\infty} D(u_n,u_n).
	\]
	
	Hence, the functional is lower semicontinuous
	\[
	E_q(u,\Omega)\le  \liminf_{n\to+\infty} E_q(u_n,\Omega),
	\]
	and $u$ is a minimizer.
\end{proof}
Having ensured the existence of an optimal function, the next step is to derive its optimality conditions. So now we study the associated Euler-Lagrange equation, analyzing some of its consequences.	
\begin{theorem}\label{thm: pde and L infty estimate}
    Let $q\in[0,1)$ and $\Omega$ be a quasi-open set of finite measure such that $E_q(\Omega)\leq E_1(B_1)$ . Then every optimal function $u_q\in \H(\Omega)$ of $E_q(\Omega)$ satisfies in a weak sense
    \begin{equation}\label{eq: PDE}
 -\Delta u_q-\lambda_{q,\Omega}u_q=-\frac{q}{2}\Big(u_q*\frac{1}{|\cdot|}\Big) \quad \text{in }\H(\Om), 
\end{equation}
 where
	\[ \lambda_{q,\Omega}=\int_\Omega|\nabla u_q|^2 {\, dx} +\frac{q}{2}D(u_q,u_q).\]
Furthermore $u_q\in C^\infty(\Omega)$ and \[  ||u_q||_{L^\infty(\Omega)}\leq C,\qquad \|v_{u_q}\|_{L^\infty(\R^3)}\leq C\]
for a constant $C=C(|\Omega|)$.
\end{theorem}

\begin{proof}
By Lemma \ref{le:exis u min}, let $u_q$ be an optimal function arising from the minimisation of $E_q(\Omega)$. By the direct computation of its first variation we have that it satisfies in weak sense
\begin{equation*}
 -\Delta u_q-\lambda_{q,\Omega}u_q=-\frac{q}{2}\Big(u_q*\frac{1}{|\cdot|}\Big).
\end{equation*}
Let us now call $f(x):=-\frac{q}{2}\left(u_q*\frac{1}{|\cdot|}\right)(x)$, then by \eqref{eq:Hardy-Sobolev}
\[ |f(x)|\leq C||\nabla u_q||_{L^1(\Omega)}\leq C|\Omega|^{1/2}||\nabla u_q||_{L^2(\Omega)}\leq C|\Omega|^{1/2}\max\{\lambda_{q,\Omega};1\}, \]
so $f(x)\in L^\infty(\Om)$. 
Moreover by the Faber-Krahn inequality
\[ ||u_q||_{L^2(\Omega)}\leq \frac{1}{\lambda_0(\Omega)}||\nabla u_q||_{L^2(\Om)}\leq \frac{|\Omega|^{2/3}}{\lambda_0(B)}\lambda_{q,\Omega},  \]
where $B$ is a ball of measure one.
By the definition of  $\lambda_{q,\Omega} $, it holds
\[ \lambda_{q,\Omega}= E_q(\Omega)\leq E_1(\Omega)\leq E_1(B_1), \]
so that, by classical elliptic regularity theory (see \cite[Theorem 8.8 and  Theorem 8.15]{GilbargTrudinger}), $u_q\in W^{2,2}_{\text{loc}}(\Om)$ and
\[ ||u_q||_{L^\infty(\Omega)}\leq C(|\Om|).\]
Eventually, recalling also that $-\Delta v_{u_q}=u_q$ in $\Omega$, we conclude that $\|v_{u_q}\|_{L^\infty(\R^3)}\leq C$ and by a bootstrap argument that $u_q\in C^\infty(\Omega)$. 
\end{proof}

	\begin{remark}\label{rmk:scaleinvariant}
          It is not difficult to check that the
          scale invariant functional
		\[
		\widetilde E_q(v,\Omega):=\frac{\int_{\Omega}|\nabla v|^2\, dx}{\|v\|_{L^2}^2}
		+\frac{q}{2}\frac{\int_\Omega\int_\Omega\frac{v(x)v(y)}{|x-y|}\,dx\,dy}{\|v\|_{L^2}^2},\qquad v\in H^1_0(\Omega),
		\]
		leads to an equivalent, yet unconstrained, minimization problem
		\[
		\min\Big\{E_q(v,\Omega) : v\in H^1_0(\Omega),\; \int_\Omega v^2\,
		dx=1\Big\}=\min\Big\{\widetilde E_q(v,\Omega) : v\in H^1_0(\Omega)
		{, v \not\equiv 0}\Big\}.
		\]
		In particular we can always choose an optimal function $u\in H^1_0(\Omega)$ for the unconstrained problem which satisfies, a posteriori, the constraint $\|u\|_{L^2}=1$.
	\end{remark}

Now we introduce an equivalent formulation of problem \eqref{eq: EqOmega} without any $L^2$ constraint, in a different spirit with respect to Remark \ref{rmk:scaleinvariant}, and which is more suited for the regularity arguments which will follow in next sections. We recall that  $\Omega\subset \R^3$ is quasi-open set with finite measure, then we define 
	\begin{equation}\label{eq:EqM}
E_{q,M}(\Om):=	\min\Big\{E_{q,M}(v,\Omega) : v\in H^1_0(\Omega)\Big\},
\end{equation}
where 	\begin{equation}\label{eq:EqM(u, Omega)}
	E_{q,M}(v,\Omega):=E_q(v,\Omega)+M\Big|\int_\Om v^2\, dx-1\Big|.
\end{equation}
Reasoning exactly as in Lemma \ref{le:exis u min} we can prove that there exists a minimizer for $E_{q,M}(\Omega)$. We show now that for some value of $M$, problems \eqref{eq:EqM} and \eqref{eq: EqOmega} are equivalent.
\begin{proposition}\label{prop:equivalence E_q and E_q,M}
	Let $\Omega\subseteq \R^3$ be a quasi-open set with finite measure.  Assume that  $M \geq 5 E_{q,M}(\Omega)+5E_1(B_1)$. Then the minimizers of the problem~\eqref{eq: EqOmega} are the same as those of \eqref{eq:EqM}.
\end{proposition}

\begin{proof}
First of all, it is easy to check that $E_{q,M}(\Omega)\leq E_q(\Omega)$, since optimal functions for the latter are admissible in the minimization of the former and are $L^2$-normalized.
Now, let $\hat{u}$ be an optimal function for $E_{q,M}(\Omega)$. Observe that $||\hat{u}||_{L^2(\Omega)}\geq 1/2$. Indeed if by contradiction $||\hat{u}||_{L^2(\Omega)}< 1/2$, then (using also the assumption on $M$)
	\begin{align*}
		E_{q,M}(\Omega)=E_{q,M}(\Omega,\hat{u})=E_q(\hat{u},\Omega)+M\Big(1-\int_\Om \hat{u}^2dx\Big)> \frac{3}{4}M\geq E_1(B_1)+E_{q,M}(\Omega)\geq E_{q,M}(\Omega).
  	\end{align*}    
Let us now call $\sigma:=||\hat{u}||_{L^2(\Om)}-1\geq -1/2$ and $u:=\frac{\hat{u}}{1+\sigma}$. We can then compute
	\begin{align*}
		E_q(\Om)\geq E_{q,M}(\Omega)= E_{q,M}(\hat{u},\Om)&=(1+\sigma)^2\int_\Om|\nabla u|^2dx+\frac{q}{2}(1+\sigma)^2D(u,u)+M\left| (1+\sigma)^2-1 \right|\\
		&\geq E_q(u,\Om)+2\sigma \Big(\int_\Om|\nabla u|^2dx+\frac{q}{2}D(u,u)\Big)+M|2\sigma+\sigma^2|
    \end{align*}
At this point, if $\sigma>0$ we immediately find a contradiction for all $M>0$.
On the other hand, if $\sigma\in (-1/2,0)$, we notice that $|2\sigma+\sigma^2|\geq |\sigma|$ and then
    \begin{align*}
		E_q(\Om)&\geq  E_q(\Om)+2\sigma E_q(u,\Om)+M|\sigma|,
	\end{align*}
so we obtain a contradiction since $E_q(u,\Om)\leq 4E_{q,M}(\Om)$ (recall that $(1+\sigma)^2\geq 1/4$) and $M\geq 5E_{q,M}(\Om)$. Thus $\sigma=0$.
\end{proof}
\begin{remark}\label{rmk:M}
	From now on, we fix once and for all a constant
    \[
    M>10 (E_1(B_1)+|B_1|).
    \]
    Clearly, if $\Om$ is a quasi-open set of finite measure such that $E_{q,M}(\Omega)\leq E_1(B_1)+|B_1|$, then $M> 10 (E_1(B_1)+|B_1|)\geq 5E_{q,M}(\Omega)+5 E_1(B_1)$ and so by Proposition \ref{prop:equivalence E_q and E_q,M}, we obtain that problems $E_q(\Om)$ and $E_{q,M}(\Om)$ are equivalent. 
    Now let us focus on the case when $q\in (0,1)$, when the monotonicity of the functionals with respect to $q$ entails that
	\[ 	\inf\{E_{q,M}(\Omega): \Om\subseteq \R^3 \text{ quasi-open},\;|\Om|=|B_1|\}\leq\inf\{E_{q}(\Omega): \Om\subseteq \R^3 \text{ quasi-open},\;|\Om|=|B_1|\}\leq E_1(B_1),\]
	thus considering in the above minimizations only quasi-open sets $\Omega$ with $E_{q,M}(\Omega)\leq E_1(B_1)+|B_1|$ is not restrictive.
    In conclusion, for $q\in(0,1)$ and $M$ chosen as above, 
	\[  \inf\{E_q(\Omega): \Om\subseteq \R^3 \text{ quasi-open},\;|\Om|=|B_1|\}=\inf\{E_{q,M}(\Omega): \Om\subseteq \R^3 \text{ quasi-open},\;|\Om|=|B_1|\}.\]
\end{remark}

\section{A surgery result}\label{sec:surgery}
    In this section, we prove a surgery result that  allows to remove the equiboundedness assumption.  The surgery strategy that we
	employ is similar to the one proposed in~\cite{MazzoleniPratelli} (see
	also~\cite{BucurMazzoleni}) and used for problems more similar to the current one in
	in~\cite{MazzoleniRuffini, MazzoleniMuratovRuffini}.   Nevertheless some non-trivial technical changes are needed to adapt it to our case, whence we report the full proof here below. In particular some additional technicalities are due to the PDE solved by optimal functions and to the fact that they can change sign.
    \begin{lemma}\label{le:surgery}
          There exist universal constants $D$, $\overline \delta<1$
          and $q_1\in (0,1)$ such that if
          $q\leq q_1$ then for any open and connected set
          $\Om\subset \R^3$ of measure $|B_1|$ satisfying
          $E_q(\Omega)-\lambda_0(B)\leq \overline \delta$ there
          exists an open, connected set $\widehat \Om$ of measure
          $|B_1|$ with diameter bounded by $D$ and such that
		\[
		E_q(\widehat \Om)\leq E_q(\Om).
		\]
	\end{lemma}

We show this result in several steps. Let us introduce some notation.  Let $\Omega$ be a connected set of
	measure $|B_1|$ such that
	$\lambda_0(\Om)-\lambda_0(B_1)\leq E_q(\Omega)-\lambda_0(B_1) \leq \overline
	\delta$, with $\overline\delta\in(0,1)$ to be chosen, and { fix $B_1$ the ball attaining the minimum in the Fraenkel asymmetry for $\Omega$ (see~\eqref{eq: fraenkel}). We assume, up to a translation of $\Omega$, that $B_1$ is centered at the origin. Then,} by the quantitative Faber--Krahn inequality (see
	Theorem~\ref{thm:quantitativefk}), we have 
	\[ |\Om\Delta B_1|=\mathcal A(\Om)\leq |B_1|^
	{1/3}\left(\frac{\overline \delta}{\widehat \sigma}\right)^{1/2},
	\]
	where $\widehat \sigma$ is the constant provided by Theorem~\ref{thm:quantitativefk}.
	By defining 
		\begin{equation}\label{eq:defK}
			K:=\lambda_0(B_1)+1\geq \lambda_0(B_1)+\overline \delta
		\end{equation} 
		we obtain immediately 
		\[
		E_q(\Omega)\leq K,\qquad\text{and in particular,}\qquad \int_{\Omega}|\nabla u|^2\, dx\leq K,
		\] 
	where $u=u_{q,\Omega}$ from now on is the function attaining $E_q(\Omega)$.
	We then note that (since $B_1$ has unit radius)\[
	|\Om\setminus [- t,t]^3|\leq |\Om\Delta B_1|= \mathcal A(\Om),\qquad \text{for all }t\geq 1.
	\]
	Let $\mh\in (0, 1/4)$ be such that
	\begin{equation}\label{defmh}
		\frac{(4\mh)^{\frac 23}}{\lambda_0(B_1)|B_1|^{\frac23}} \, K \leq \frac 12\,.
	\end{equation}
	Moreover, we choose $\overline \delta$ small enough so that 
	\begin{equation}\label{eq:asimmpiccolamhat}
		|\Om\setminus [-1 ,1 ]^3|\leq \mathcal A(\Om)\leq |B_1|\left(\frac{\overline \delta}{\widehat \sigma}\right)\leq\frac{\mh}{2^{6}}.
	\end{equation}
	We first focus on the direction $e_1$ and detail the construction in this case.
	We shall denote $z=(x,y)\in \R\times\R^{2}$ and by $z_i$ the $i$-th component of $z\in \R^3$.
	For any $t\in \R$, we  define
	\begin{equation*}
		\Om_t:=\Big\{y\in\R^{2} : (t,y)\in\Om\Big\}\,,
	\end{equation*}
	and given any set $\Omega\subseteq\R^3$, we define its $1$-dimensional projections for $ p\in\{1,2,3\}$ as
	\[
	\pi_p(\Omega) := \Big\{ t\in\R:\, \exists \, (z_1,\, z_2, z_3)\in\Omega,\, z_p = t\Big\}\,.
	\]
	For every $t\leq -1$ we call
	\begin{align}\label{int0}
		\Omega^+(t) := \Big\{(x,y)\in\Om : x>t\Big\}\,, && \Omega^-(t) := \Big\{(x,y)\in\Om : x<t\Big\}\,, && \eps(t):=\mathcal H^{2}(\Om_t)\,.
	\end{align}
	Observe that by \eqref{eq:asimmpiccolamhat}
	\begin{equation}\label{int1}
		m(t) := \big| \Omega^-(t) \big| = \int_{-\infty}^t \eps(s)\,ds\leq 2\mh\,.
	\end{equation}
	We call $u$ a optimizer for
		$E_q(\Omega)$. We define then also, for every $t\leq -1$,
	\begin{align}\label{int2}
		\de(t):=\int_{\Om_t}{|\nabla u(t,y)|^2\,d\hc^{2}(y)}\,, && \mu(t):=\int_{\Om_t}{u(t,y)^2\,d\hc^{2}(y)}\,, 
	\end{align}
	which makes sense since $u$ is smooth inside $\Om$, see Theorem~\ref{thm: pde and L infty estimate}. 
	Applying the Faber--Krahn inequality in $\R^{2}$ to the set $\Omega_t$, and using the rescaling property of eigenvalues on $\R^{2}$, we know that
	\[
	\eps(t) \lambda_0(\Omega_t)=\hc^{2}(\Omega_t) \lambda_0(\Omega_t) \geq \lambda_0(B_{\R^2})\,,
	\]
	{calling $B_{\R^2}$ the ball of unit measure in $\R^{2}$.} As a trivial consequence, we can estimate $\mu$ in
	terms of $\eps$ and $\de$: in fact, noting that
	$u(t,\cdot)\in H^{1}_0(\Omega_t)$ and writing
	$\nabla u = (\nabla_1 u, \nabla_y u)$, we have
	\begin{equation}\label{eq:muest}
		\mu(t)=\int_{\Om_t}{u(t,\cdot)^2\,d\hc^{2}}\leq \frac{1}{\lambda_0(\Om_t)}\int_{\Om_t}{|\nabla_y u(t,\cdot)|^2\,d\hc^{2}}\leq C\eps(t)\de(t).
	\end{equation}
	We can now present two estimates which assure that $u$ and $\nabla u$
	cannot be too big in $\Om^-(t)$.
	\begin{lemma}\label{primastima}
		{Let $\Omega\subseteq \R^3$ and $u$ be as in Lemma~\ref{le:surgery}.} For every $t\leq -1$ the following inequalities hold: 
		\begin{align}\label{eq:udu-}
			\int_{\Omega^-(t)} u^2 \, dx\leq C_1 \eps(t)^{\frac{1}{2}}\de(t)+q C_m m(t)\,, &&
			\int_{\Omega^-(t)} |\nabla u|^2 \, dx\leq C_1 \eps(t)^{\frac{1}{2}}\de(t)+q C_m m(t)\,,
		\end{align}
		for some universal constants $C_1, C_m {>0}$ .
	\end{lemma}
The proof of the above Lemma follows, up to a few minor changes, as in~\cite[Lemma~2.3]{MazzoleniPratelli}, by working on $u$. 
The main difference is that $u$ solves the PDE
\begin{equation}\label{eq:PDEq}
\begin{cases}
-\Delta u=\lambda_q u- \frac{q}{2} v_u,\qquad&\text{in }\Omega,\\
u=0,\qquad &\text{on }\partial \Omega,
\end{cases}
\end{equation}
instead of being first eigenfunction of the Dirichlet Laplacian in $\Omega$. 
In particular, if $u\geq 0$, we note that $u$ solves the differential inequality $-\Delta u\leq \lambda_q u$, in $\Omega$ and things work as in~\cite{MazzoleniMuratovRuffini} without the need of the additional term $C_m m(t)$.
We reproduce fully the proof here for the sake of completeness.
		\begin{proof}
			Let us fix $t\leq -1$. Consider the set $\Om_S^-$ obtained by the union of $\Om^-(t)$ and its reflection with respect to the plane $\{x=t\}$, and call $u_S \in H^{1}_0(\Om_S)$ the function obtained by reflecting $u$. Using the Faber-Krahn inequality, we find then
			\[
			\frac{\lambda_0(B_1)|B_1|^{\frac23}}{\big(2m(t)\big)^{\frac{2}{3}}} = \frac{\lambda_0(B_1)|B_1|^{\frac 23}}{|\Om_S^-|^{\frac{2}{3}}}
			\leq \lambda_0(\Om_S^-) \leq \frac{\begin{aligned}\int_{\Om^-_S}|\nabla u_S|^2\,dx\end{aligned}}{\begin{aligned}\int_{\Om^-_S}u_S^2\,dx\end{aligned}}= \frac{\begin{aligned}\int_{\Om^-(t)}|\nabla u|^2\,dx\end{aligned}}{\begin{aligned}\int_{\Om^-(t)}u^2\,dx\end{aligned}}
			\]
			by the symmetry of $\Om^-_S$, and using the scaling. This estimate gives
			\begin{equation}\label{eq:eq1}
				\int_{\Om^-(t)} u^2\,dx \leq \frac{\big(2m(t)\big)^{\frac 23}}{\lambda_0(B_1)|B_1|^{\frac23}}\, \int_{\Om^-(t)} |\nabla u|^2\,dx
			\end{equation}
			which in particular, being $m(t)\leq 2\mh$ and recalling~\eqref{defmh}, implies
			\begin{equation}\label{estray2}
				\int_{\Omega^-(t)} u^2\,dx \leq \frac 12\,.
			\end{equation}
			On the other hand, recalling~\eqref{eq:PDEq}, by Schwarz inequality, the uniform $L^\infty$ bound on $v_u$ and $u$, recalling $m(t)=|\Omega^-(t)|$ and using~\eqref{eq:muest} we have 
			\begin{equation}\label{eq:eq2}\begin{split}
					\int_{\Om^-(t)} |\nabla u|^2\,dx &= \int_{\Om^-(t)} \lambda_q u^2-\frac{q}{2}v_u u\,dx + \int_{\Om_t}  u \,\frac{\partial u}{\partial \nu}\,d\mathcal H^2\\
					&\leq K \int_{\Om^-(t)} u^2\,dx +q C_m m(t)+ \sqrt{\int_{\Om_t} u^2\,d\mathcal H^2  \int_{\Om_t} |\nabla u|^2\,d\mathcal H^2}\\
					&\leq K \int_{\Om^-(t)} u^2\,dx + C \eps(t)^{\frac 1{2}} \delta(t)+q C_m m(t)\,.
			\end{split}\end{equation}
			It is now easy to obtain~\eqref{eq:udu-} combining~\eqref{eq:eq1} and~\eqref{eq:eq2}. In fact, by inserting the latter into the first, we find
			\[
			\int_{\Om^-(t)} u^2\,dx \leq  \frac{\big(2m(t)\big)^{\frac 23}}{\lambda_0(B_1)|B_1|^{\frac23}}\, \bigg( K \int_{\Om^-(t)} u^2\,dx + C \eps(t)^{\frac 1{2}} \delta(t)+q C_m m(t) \bigg)\,,
			\]
			which by~(\ref{defmh}) again yields
			\begin{equation}\label{eq:leftudu-}
				\frac{1}{2} \int_{\Om^-(t)} u^2\,dx \leq  \frac{\big(2m(t)\big)^{\frac 23}}{\lambda_0(B_1)|B_1|^{\frac23}}\,\left[ C \eps(t)^{\frac 1{2}} \delta(t) +q C_m m(t)\right]
				\leq C \eps(t)^{\frac 1{2}} \delta(t)+q C_m m(t) \,.
			\end{equation}
			The left estimate in~\eqref{eq:udu-} is then obtained. To obtain the right one, one has then just to insert~\eqref{eq:leftudu-} into~\eqref{eq:eq2}.
		\end{proof}
	
	Let us go further into the construction, giving some additional definitions. For any $t\leq -1$ and $\sigma(t)>0$, we define the cylinder $Q(t)$ as
	\begin{equation}\label{defcyl}
		Q(t):=\Big\{(x,y)\in \R^3: \, t-\sigma(t) < x < t,\ (t,y) \in\Om\Big\} = \big(t-\sigma(t),t\big) \times \Omega_t\,,
	\end{equation}
	where for any $t\leq -1$ we set
	\begin{equation}\label{defsigma}
		\sigma(t)= \eps(t)^{\frac 1{2}}\,.
	\end{equation}
	We let also $\Omt(t)=\Omega^+(t)\cup Q(t)$, and we introduce $\ut\in H^1_0\big(\Omt(t)\big)$ as
	\begin{equation}\label{utilde}
		\ut(x,y):=\left\{
		\begin{array}{ll}
			u(x,y) &\hbox{if $(x,y)\in \Omega^+(t)$}\,, \\[5pt]
			\begin{aligned}
				\frac{x-t+\sigma(t)}{\sigma(t)}\,u(t,y)
			\end{aligned}&\hbox{if $(x,y)\in Q(t)$}\,.
		\end{array}
		\right.
	\end{equation}
	The fact that $\ut$ vanishes on $\partial\Omt(t)$ is obvious; moreover, $\nabla u=\nabla \ut$ on $\Omega^+(t)$, while on $Q(t)$ one has
	\begin{equation}\label{estdut}
		\nabla \ut(x,y)=\left(\frac{u(t,y)}{\sigma(t)}\,,\,\frac{x-t+\sigma(t)}{\sigma(t)}\,\nabla_y u(t,y)\right)\,.
	\end{equation}
	
	A simple calculation allows us to estimate the integrals of $\ut$ and $\nabla \ut$ on $Q(t)$.
	\begin{lemma}\label{lemmatest}
		For every $t\leq -1$, one has
		\begin{align}\label{newtest}
			\int_{Q(t)}|\nabla \ut|^2 \, dx\leq C_2\eps(t)^{\frac{1}{2}}\de(t) \,, &&
			\int_{Q(t)} \ut^2 \, dx\leq C_2 \eps(t)^{\frac 3{2}} \delta(t)\,,
		\end{align}
		for a universal constant $C_2>0$.
	\end{lemma}
	The proof of the above Lemma follows as~\cite[Lemma~2.4]{MazzoleniPratelli}, as it does not depend on the PDE solved by $u$.

	Another simple but useful estimate concerns the Rayleigh quotients of the functions $\ut$ on the sets $\Omt(t)$: notice that, while $u$ has unit $L^2$ norm, the modifed function $\ut$ in general is not normalized so we need to take care also of its norm.
	
	\begin{lemma}\label{noth}
		There exists a universal constant $C_3>0$ such that for every $t\leq -1$, one has
		\begin{equation}\label{est1a}
			\int_{\Omt(t)}|\nabla \ut|^2\, dx\leq \int_\Omega|\nabla u|^2\, dx+ C_3\eps(t)^{\frac 1 {2}}\de(t)\,,\qquad \int_{\Omt(t)}\ut^2\, dx\geq \int_\Omega u^2\, dx-C_3\eps(t)^{\frac 1 {2}}\de(t)-q C_m m(t)\,.
		\end{equation}
	\end{lemma}
	\begin{proof}
		It is enough to note that, by definition of $\Omt(t)$ and using Lemma~\ref{primastima} and~\ref{lemmatest}, we obtain for the gradient term\[
		\begin{split}
			\int_{\Omt(t)}|\nabla \ut|^2\, dx&=\int_{\Om^+(t)}|\nabla u|^2\, dx+\int_{Q(t)}|\nabla \ut|^2\, dx\\
			&=\int_\Omega|\nabla u|^2\, dx+\int_{Q(t)}|\nabla \ut|^2\, dx-\int_{\Om^-(t)}|\nabla u|^2\, dx\leq \int_\Omega|\nabla u|^2\, dx+C_2\eps(t)^{\frac 1 {2}}\de(t)\,,
		\end{split}
		\]
		while for the function, we have \[
		\int_{\Omt(t)} \ut^2\, dx=\int_{\Om^+(t)} u^2\, dx+\int_{Q(t)} \ut^2\, dx=\int_\Omega u^2\, dx+\int_{Q(t)}\ut^2\, dx-\int_{\Om^-(t)} u^2\, dx\geq \int_\Omega u^2\, dx-C_1\eps(t)^{\frac 1 {2}}\de(t)-q C_m m(t)\,.
		\]
	\end{proof}

	We can now enter in the central part of our construction. Basically, we aim to show that either $\Omega$ already has bounded left ``tail'' in direction $e_1$, or some rescaling of $\Omt(t)$ has energy lower than that of $\Omega$. 
	\begin{lemma}\label{threeconditions}
          Let $\Omega$ be as in the assumptions of
          Lemma~\ref{primastima}, and let $t\leq -1$. There exist
          universal $ q_1\in(0,1)$ and $C_4>2$ such
          that, for all $q\leq q_1$ exactly one of the three
          following conditions hold:
		\begin{enumerate}
			\item[(1) ] $\max\big\{ \eps(t),\, \delta(t) \big\} > 1$;
			\item[(2) ] (1) does not hold and $m(t) \leq C_4 \big( \eps(t) + \delta(t)\big) \eps(t)^{\frac 1{2}}$;
			\item[(3) ] (1) and~(2) do not hold and there holds
            \[
			\frac{\int_{\Omh(t)}|\nabla \uh|^2\, dx}{\int_{\Omh(t)}\uh^2\, dx}\leq \int_\Omega|\nabla u|^2\, dx,\qquad \text{and}\qquad
			E_q\big(\Omh(t)\big)< E_q(\Omega),
			\] where for $t\leq -1$  we set
			\[
			\Omh(t) := \big|B_1\big|^{\frac13}\big| \Omt(t) \big|^{-\frac 13} \Omt(t),\qquad \text{and} \qquad \uh(x)=\ut\big(|B_1|^{-\frac13}|\Omt(t)|^{\frac13} x \big),\quad\text{for $x\in \Omh(t)$.}
			\]
		\end{enumerate}
	\end{lemma}
	\begin{proof}
		Assume~(1) is false. Then it is possible to apply Lemma~\ref{noth}, to obtain
\begin{equation}\label{putinto}
    		\begin{gathered}
			\int_{\Omt(t)}|\nabla \ut|^2\, dx\leq \int_\Omega|\nabla u|^2\, dx+ C_3\eps(t)^{\frac 1 {2}}\de(t)\,,\\ \int_{\Omt(t)}\ut^2\, dx\geq \int_\Omega u^2\, dx-C_3\eps(t)^{\frac 1 {2}}\de(t)-q C_m m(t)=1-C_3\eps(t)^{\frac 1 {2}}\de(t)-q C_m m(t)\,.
		\end{gathered}
\end{equation}
		By the scaling properties of the eigenvalue and the fact that $\big| \Omh(t)\big|=|B_1|$, we know that
		\[
		\frac{\int_{\Omh(t)}|\nabla \uh|^2\, dx}{\int_{\Omh(t)}\uh^2\, dx}= \frac{\big| \Omt(t) \big|^{\frac 23}}{|B_1|^{\frac23}} 
		\frac{\int_{\Omt(t)}|\nabla \ut|^2\, dx}{\int_{\Omt(t)}\ut^2\, dx}\,.
		\]
		By construction,
		\[
		\big| \Omt(t)\big|= \big| \Om^+(t)\big|+ \big|Q(t)\big| = |B_1| - m(t) + \eps(t)^{\frac 3{2}}\,,
		\]
		hence the above estimates, the scaling of the integrals due to the definition of $\uh$ and~\eqref{putinto} lead to
		\begin{equation}\label{muchhere}\begin{split}
				\frac{\int_{\Omh(t)}|\nabla \uh|^2\, dx}{\int_{\Omh(t)}\uh^2\, dx} &= \Big( 1 - \frac{m(t)}{|B_1|} + \frac{\eps(t)^{\frac32}}{|B_1|} \Big)^{\frac 23}\, \frac{\int_{\Omt(t)}|\nabla \ut|^2\, dx}{\int_{\Omt(t)}\ut^2\, dx},  \\
				&\leq\Big( 1 - \frac{2}{3|B_1|} \, m(t) +  \frac{2}{3|B_1|}\, \eps(t)^{\frac 3{2}} \Big) \Big(1+C_3\eps^{\frac{1}{2}}(t)\delta(t)+q C_m m(t)\Big)  \Big( \int_\Omega|\nabla u|^2\, dx+C_3\eps(t)^{\frac 1 {2}}\delta(t) \Big)\\
				&\leq \int_\Omega|\nabla u|^2\, dx - \frac{2\lambda_0(B_1)}{3|B_1|}\, m(t)+q K C_m m(t)+ \frac{2K}{3|B_1|}\, \eps(t)^{\frac 3{2}}+\bigg(2C_3+KC_3+ \frac{2}{3|B_1|}\bigg)\eps(t)^{\frac 1 {2}}\delta(t)\,\\
				&\leq \int_\Omega|\nabla u|^2\, dx - \frac{49}{100}\pi\, m(t)+ \frac{2K}{3|B_1|}\, \eps(t)^{\frac 3{2}}+\bigg(2C_3+KC_3+ \frac{2}{3|B_1|}\bigg)\eps(t)^{\frac 1 {2}}\delta(t)\, ,
			\end{split}
		\end{equation}
 noting that (in three dimension) $\frac{\pi}{2}=\frac{2\lambda_0(B_1)}{3|B_1|}$ and up to take $ q_1$ such that \[
q_1 K C_m<\frac{\pi}{100}=\frac{2\lambda_0(B_1)}{150|B_1|}.
\]	
		
		At this point, defining $C_4:= \max{\{\frac{2(K+1)}{3|B_1|}+2C_3+KC_3,2\}}$, if \[
		m(t) \leq C_4 \big( \eps(t) + \delta(t)\big) \eps(t)^{\frac 1{2}},
		\] 
		then condition~(2) holds true. 
		Otherwise, we immediately have that 
		\begin{equation}\label{eq:stimala1}
			\frac{\int_{\Omh(t)}|\nabla \uh|^2\, dx}{\int_{\Omh(t)}\uh^2\, dx} \leq \int_\Omega|\nabla u|^2\,dx-\left(\frac{49}{100}\pi-1\right)m(t)\leq \int_\Omega|\nabla u|^2\, dx-C_5m(t),
		\end{equation}
		for a universal constant $C_5>0$, therefore the first part of the third claim is verified.
		
		On the other hand, we note that, using the $L^\infty$ bound of $u$, see Theorem~\ref{thm: pde and L infty estimate}, the fact that $\|\ut\|_{L^\infty}\leq \|u\|_{L^\infty}$ by construction and also by~\cite[Lemma~2.4]{FuscoPratelli}, 		
		\[
		\begin{split}
		D(\ut,\ut)&= D(u,u)-2\int_{\Om^-(t)}\int_{\Om^+(t)}\frac{u(x)u(y)}{|x-y|}\,dxdy-\int_{\Om^-(t)}\int_{\Om^-(t)}\frac{u(x)u(y)}{|x-y|}\,dxdy\\
		&\hspace{3cm}+2\int_{\Om^+(t)}\int_{Q(t)}\frac{\ut(x)\ut(y)}{|x-y|}\,dxdy+\int_{Q(t)}\int_{Q(t)}\frac{\ut(x)\ut(y)}{|x-y|}\,dxdy\\
		&\leq D(u,u)+\tilde C_m m(t)+ C_{fp}\eps^{\frac{3}{2}}(t),
		\end{split}
		\]
		where $\tilde C_m$ and $C_{fp}$ are positive universal constants.
		Then we can estimate, using the appropriate scalings,  
		\begin{equation}\label{eq:stimaV}
			\begin{split}
				&\frac{D(\uh,\uh)}{\int_{\Omh(t)}\uh^2\, dx}\leq \frac{D(\ut,\ut)}{\int_{\Omt(t)}\ut^2\, dx}\left(1-\frac{m(t)}{|B_1|}+\frac{\eps(t)^{\frac{3}{2}}}{|B_1|}\right)^{-\frac{2}{3}}\leq \Big(1+\frac{2}{3|B_1|}m(t)\Big)\frac{D(\ut,\ut)}{\int_{\Omt(t)}\ut^2\, dx}\\
				&\leq  \Big(1+\frac{2}{3|B_1|}m(t)\Big)\Big(1+C_3\eps^{\frac{1}{2}}(t)\delta(t)\Big)\Big(D(u,u)+\tilde C_m m(t)+C_{fp}\eps^{\frac{3}{2}}(t)\Big)\\
				&\leq D(u,u) +C\|u\|^2_{L^\infty}m(t)+\tilde C_m m(t)+C_{fp}\eps^{\frac{3}{2}}(t)+C_3\|u\|^2_{L^\infty}\eps^{\frac{1}{2}}(t)\delta(t)\\
				&\leq D(u,u) +C\|u\|^2_{L^\infty}m(t)+\tilde C_m m(t)+(C_{fp}+C_3\|u\|^2_{L^\infty})m(t)\\
				&=D(u,u)+C_6 m(t).
			\end{split}
		\end{equation}
		Then, putting together~\eqref{eq:stimala1} and~\eqref{eq:stimaV}, recalling also Remark~\ref{rmk:scaleinvariant} for the equivalence of the scale invariant energy, 
		\begin{equation}
			\begin{split}
				& E_q(\Omh(t))\leq  \frac{\int_{\Omh(t)}|\nabla \uh|^2\, dx}{\int_{\Omh(t)}\uh^2\, dx} +\frac{q}{2} \frac{D(\uh,\uh)}{\int_{\Omh(t)}\uh^2\, dx}\\
				&\leq \int_\Omega|\nabla u|^2\, dx+\frac{q}{2} D(u,u)-(C_5-\frac{q}{2} C_6)m(t)\\
				&\leq \int_\Omega|\nabla u|^2\, dx+\frac{q}{2} D(u,u)-\frac{C_5}{2}m(t),
			\end{split}
		\end{equation}
		up to taking $q\leq q_1<\min\left\{\sqrt{\frac{C_5}{2C_6}}, \frac{\pi}{100 K C_m}\right\}$, so that in
		this case condition $(3)$ holds and the proof is concluded.
	\end{proof}
	
	Once we have Lemma~\ref{threeconditions}, the rest of the proof follows as in~\cite{MazzoleniPratelli} or~\cite{MazzoleniRuffini} as we detail here below.
	
	\begin{proof}[Proof of Lemma~\ref{le:surgery}]
		It is enough to repeat the {analogs} of~\cite[Lemma~8.7,
		Lemma~8.8, Proposition~8.1 and Section~9.2]{MazzoleniRuffini}, noting that it is
		only a geometric argument and having $\int_\Omega|\nabla u|^2\, dx$
		instead of $\lambda_0(\Omega)$ does not change anything.
	\end{proof}

\section{The auxiliary problem}\label{sec:Auxiliary}
To avoid the restriction imposed by the measure constraint, we employ the strategy first suggested by Aguilera, Alt, and Caffarelli \cite{AguileraAltCaffarelli} adding the following penalization term.	Let $\eta\in(0,1)$ and consider the piecewise linear function 
\[
f_\eta\colon \R^+\rightarrow \R,\qquad f_\eta(s)=
\begin{cases}
\eta(s-|B_1|),\qquad \text{if }s\leq |B_1|,\\						
	\frac{1}{\eta}(s-|B_1|),\qquad \text{if }s\geq |B_1|.
\end{cases}
\]
It is easy to check that, for all $0\leq s_2\leq s_1,$ there holds
\begin{equation}\label{eq:prop feta}
	\eta(s_1-s_2)\leq f_\eta(s_1)-f_\eta(s_2)\leq \frac{1}{\eta}(	s_1-s_2).
\end{equation}

Let us now consider the problem 
\begin{equation}\label{eq:def of inf E_{q,M,eta}(Omega,u) global}
	\inf_{\substack{\Omega\subseteq \R^3\\\text{quasi-open}}}\inf_{u\in \H(\Om)} E_{q,M,\eta}(\Om,u). 
\end{equation}
where 
\begin{equation}
E_{q,M,\eta}(\Om,u)= E_{q,M}(u,\Omega)+f_\eta(|\Om|).
\end{equation}

Even without knowing at this point whether  existence of an optimal set for problem \eqref{eq:def of inf E_{q,M,eta}(Omega,u) global} holds true, we point out that we can at least select one \textit{good} minimizing sequence. 
\begin{lemma}\label{le:good minizing sequence for surgery}
    Let $q\in [0,1)$. Then there exists a universal constant $\eta_1>0$ such that for all $\eta\in(0,\eta_1)$ there exists a minimizing sequence $(\Omega_n)_n$ of problem  \eqref{eq:def of inf E_{q,M,eta}(Omega,u) global} such that $\Omega_n$ are connected and
    \[  C\leq |\Omega_n|\leq |B_2|\]
    for a universal constant $C$.
\end{lemma}
\begin{proof}
    	Let us suppose for the sake of contradiction that there exists a minimizing sequence $(\Omega_n)_n$ of problem  \eqref{eq:def of inf E_{q,M,eta}(Omega,u) global} such that $|\Om_{n}|> |B_2|$ for all $n\in \N$. We are then going to reach a contradiction as long as
          \[ 1/\eta\ge E_{1}(B_1).\] 
	Indeed it holds	 
       \[E_{q,M,\eta}(\Om_n)\leq \inf\{E_{q,M,\eta}(\Omega):\Omega\subseteq \R^3, \text{quasi-open} \}+1\leq E_{1}(B_1)+1,	\]
		On the other hand, by Lemma \ref{le:positivityD(u,u)}, since $|\Om_n|> |B_2|$ we have
		\[  E_{q,M,\eta}(\Om_{n})\geq \frac{1}{\eta}(|\Om_{n}|-|B_1|)\geq \frac{1}{\eta} (|B_2|-|B_1|). \]
		By choosing $\eta_1$ such that $\eta_1<1$ and
		\[ \frac{(|B_2|-|B_1|)}{\eta_1}>E_{1}(B_1)+1, \]
		we reach the desired contradiction. 
        
        Let us now fix 
        \[
        \widetilde{R}=\left(\frac{\lambda_0(B_1)}{2(1+E_1(B_1)+|B_1|)}\right)^{1/2}
        \]
        and suppose by contradiction that there exists a minimizing sequence $(\Omega_n)_n$ of problem  \eqref{eq:def of inf E_{q,M,eta}(Omega,u) global} such that $|\Om_{n}|< |B_{\widetilde{R}}|$. Then, by similar computations as before, by the Faber-Krahn inequality and the monotonocity of the first eigenvalue it holds
        \[ \lambda_0(B_{\widetilde{R}})+f_\eta(|\Omega_n|) \leq \lambda_0(\Omega_n)+f_\eta(|\Omega_n|) \leq E_{q,M,\eta}(\Omega_n)\leq E_1(B_1)+1.\]
        Since $f_\eta(|\Omega_n|)\geq -\eta|B_1|$,  $\lambda_0(B_{\widetilde{R}})=\frac{1}{\widetilde{R}^2} \lambda_0(B_1)$ and $\eta\in (0,1)$, we deduce
        \[ \widetilde{R}^2\geq \frac{\lambda_0(B_1)}{1+E_1(B_1)+|B_1|}\]
        which is a contradiction with the definition of $\widetilde{R}$. Then we can take $C=|B_{\widetilde{R}}|$ in the statement.
        
        We focus now on the connectedness. Let $\Omega\subseteq \R^3$ be a term of a minimizing sequence such that $C\leq|\Omega|\leq |B_2|$, $E_{q,M,\eta}(\Omega)\leq E_1(B_1)+|B_1|$ and made up of at most countably many connected components 
        \[ \Omega= \bigcup_{k\in \N}\Omega^k.\]
        For all $\vartheta\in (0,1)$, we consider a segment $S_{k}$ connecting the components $\Omega^k$ and $\Omega^{k+1}$ and we consider $T_{k,\vartheta}=\cup_{x\in S_k}B_{\zeta}(x)$, choosing $\zeta$ so that $|T_{k,\vartheta}|\leq \frac{\vartheta}{2^k}$.
        We call now \[ \Omega_\vartheta:=\bigcup_{k\in \N}T_{k,\vartheta}\cup \Omega\supset \Omega,\qquad {\widehat \Omega}_{\vartheta}:=\left(\frac{|\Omega|}{|\Omega_\vartheta|}\right)^{1/3}\Omega_\vartheta,\]
        and note that \[ |\Omega_\vartheta|\leq  |\Omega|+\vartheta,\qquad |\widehat{\Omega}_\vartheta|=|\Omega|.\]
        To simplify the notation, let us set $t=\left(\frac{|\Omega|}{|\Omega_\vartheta|}\right)^{1/3}\leq 1$. Let $u\in \H(\Omega_\vartheta)$ be an optimal function for $E_q(\Omega_\vartheta)$, then $v(x)=t^{-3/2}u(x/t)\in \H(\widehat{\Omega}_\vartheta)$. It holds, by a rescaling argument, that
        \begin{align*}
            E_q(\widehat{\Omega}_\vartheta)\leq E_q(v, \widehat{\Omega}_\vartheta)&=t^{-2}\int_{\Omega_\vartheta}|\nabla u|^2dx+t^{2}\frac{q}{2}\int_{\Omega_\vartheta}\int_{\Omega_\vartheta} \frac{u(x)u(y)}{|x-y|}dxdy\\ &\leq (1+C\vartheta)\int_{\Omega_\vartheta}|\nabla u|^2dx+\frac{q}{2}\int_{\Omega_\vartheta}\int_{\Omega_\vartheta} \frac{u(x)u(y)}{|x-y|}dxdy= E_q(\Omega_\vartheta)+C\vartheta \int_{\Omega_\vartheta}|\nabla u|^2dx
        \end{align*}
        for a universal constant $C$. Observe that by set inclusion
        \[ E_q(\Omega_{\vartheta})\leq \inf\{ E_q(u,\Omega_\vartheta):u\in \H(\Omega),\;\int_\Omega u^2\,dx=1\}=\inf\{ E_q(u,\Omega):u\in \H(\Omega),\;\int_\Omega u^2\,dx=1\}=E_q(\Omega)\]
        and so
        \[ \int_{\Omega_\vartheta}|\nabla u|^2dx \leq E_q(\Omega_\vartheta)\leq E_q(\Omega)\leq E_1(B_1)+1,\]
        thus $E_q(\widehat\Omega_\vartheta)\leq E_q(\Omega)+C\vartheta$. Since
        \[E_{q,M}(\Omega)-|B_1|\leq E_{q,M,\eta}(\Omega)\leq E_1(B_1)+|B_1|,\]
        by Proposition \ref{prop:equivalence E_q and E_q,M} (see also Remark \ref{rmk:M}),  it holds $E_q(\Omega)=E_{q,M}(\Omega)$.   Recalling $|\widehat\Omega_\vartheta|=|\Omega|$, it holds
        \[ E_{q,M,\eta}(\widehat\Omega_\vartheta)\leq E_{q,M,\eta}(\Omega)+C\vartheta.\]
     By arbitrariety of $\vartheta$, by applying this procedure to all of the elements of the minimizing sequence, it is not restrictive to assume that they are  connected.
\end{proof}

We show now how the surgery argument of Section~\ref{sec:surgery} can be extended also to this unconstrained functional.

\begin{theorem}\label{thm: equivalence unbounded and bounded problem}
    Let $q\in[0,q_1)$ and $\eta\in(0,\eta_1)$. Then there exists a universal constant $R$ such that
    \begin{equation*}
	\inf\{E_{q,M,\eta}(\Om):\Omega\subseteq \R^3,\, \text{quasi-open}\}=\inf\{E_{q,M,\eta}(\Om):\Omega\subseteq B_R,\, \text{quasi-open}\}. 
\end{equation*}
\end{theorem}
	\begin{proof}[Proof of Theorem~\ref{thm: equivalence unbounded and bounded problem}]
    It is trivial that  \[\inf\{E_{q,M,\eta}(\Om):\Omega\subseteq \R^3,\, \text{quasi-open}\}\leq\inf\{E_{q,M,\eta}(\Om):\Omega\subseteq B_R,\, \text{quasi-open}\}.\] 
    By Lemma \ref{le:good minizing sequence for surgery} there exists a minimizing sequence for problem \eqref{eq:def of inf E_{q,M,eta}(Omega,u) global} made of connected sets with measure uniformly bounded from above and below, thus all the constants in the surgery argument (Lemma~\ref{le:surgery}) will depend only on these universal bounds and not on the measure of the optimal unconstrained set. Thus by Lemma \ref{le:surgery}, we can construct another minimizing sequence made of sets with uniformly bounded diameter $D$ and with the same measure. Observe that the term $f_\eta$ does not interfere since the new sets have the same volume of the corresponding ones. Taking $R=\max\{2D,10\}$, this is a minimizing sequence also for the same problem but in the box $B_R$. 
	\end{proof}

\begin{remark}\label{rem: R=2D}
Observe that the radius $R$ in the proof of Theorem \ref{thm: equivalence unbounded and bounded problem} has been taken as $R=\max\{2D,10\}$, where $D$ comes from \ref{le:surgery}. Hence, hereafter we are allowed to  consider $R$ to be fixed.    
\end{remark}

\subsection{Existence and some properties}\label{sec:firstproperties}
Due to the surgery result, we can now restrict us to work in a fixed box, a ball of radius $R$ (see Remark \ref{rem: R=2D}). Let us now consider the problem 
\begin{equation}\label{eq:min mathcal E q,M,eta}
	\inf\left\{\Ecal_{q,M,\eta}(u):u\in \H(B_R)\right\},
\end{equation}
where
\[ \Ecal_{q,M,\eta}(u)=\int_{B_R} |\nabla u(x)|^2\, dx+\frac{q}{2}\int_{B_R} \int_{B_R} \frac{u(x)u(y)}{|x-y|}\, dxdy+ M\left|\int_{B_R}u^2dx-1\right|+f_\eta(|\{u\ne 0\}|).\]
Observe that problem~\eqref{eq:min mathcal E q,M,eta} is equivalent to 
\begin{equation}\label{eq: R q,M,eta in B_R}
    \inf\{E_{q,M,\eta}(\Om):\Omega\subseteq B_R,\, \text{quasi-open}\}
\end{equation}

Indeed let $\{u_n\}_n$ be a minimizing sequence for problem~\eqref{eq:min mathcal E q,M,eta}. Then $\{u_n\ne0\}$ is a valid candidate in \eqref{eq: R q,M,eta in B_R}. Thus \[	\inf_{\substack{\Omega\subseteq B_R\\\text{quasi-open}}}\inf_{u\in \H(\Om)} E_{q,M,\eta}(\Om,u) \leq \mathcal{E}_{q,M,\eta}(u_n),\]
and by passing to the limit we obtain
\[	\inf_{\substack{\Omega\subseteq B_R\\\text{quasi-open}}}\inf_{u\in \H(\Om)} E_{q,M,\eta}(\Om,u) \leq \inf_{u\in \H(B_R)} \mathcal{E}_{q,M,\eta}(u),\]
Suppose by contradiction that there exists a quasi-open set $\widetilde{\Omega}\subseteq B_R$ for which the inequality is strict, namely $E_{q,M,\eta}(\widetilde \Omega)<\inf_{u\in \H(B_R)} \mathcal{E}_{q,M,\eta}(u)$.  Then there exists an optimal $\widetilde{u}\in \H(\widetilde{\Omega})$ for $E_{q,M}(\widetilde{\Omega})$. Since $\{\widetilde{u}\ne 0\}\subseteq \widetilde\Omega$ and $f_\eta$ is increasing, then we reach a contradiction by
\[ \mathcal{E}_{q,M,\eta}(\tilde{u})\leq E_{q,M}(\tilde{u},\widetilde\Omega)+f_\eta(|\widetilde\Omega|)< \inf_{u\in \H(B_R)} \mathcal{E}_{q,M,\eta}(u).\]

Since we are now in an equibounded setting, we can address the existence of an optimizer for problem \eqref{eq:min mathcal E q,M,eta}. We also show that (its support) has finite perimeter, and this proof, although inspired by the one proposed first in~\cite{Bucur}, needs to take care of the sing-changing nature of the functions involved in this problem.
\begin{theorem}\label{thm:existenceEq,M,eta}
	    Let $\eta\in(0,\eta_1)$ and $q\in [0,1)$.
	Then there exists a minimizer for problem~\eqref{eq:min mathcal E q,M,eta}.
	Moreover for all minimizers $u\in \H(B_R)$, the quasi-open set $\left\{u\ne 0\right\}$ has perimeter\footnote{Here we mean that $\left\{u\ne 0\right\}$ is a set of finite perimeter in the distributional sense. See \cite{Maggibook} for details. } uniformly bounded by a constant depending on $\eta$.
\end{theorem}
\begin{proof}
	Let $(u_n)_n$ be a sequence such that
	\begin{equation}\label{eq:minimizingsequenceEqMeta}
		\Ecal_{q,M,\eta}(u_n)\leq 	\inf\left\{\Ecal_{q,M,\eta}(u):u\in \H(B_R)\right\} + \frac1n.
	\end{equation}
	By the nonnegativity of the terms in $\Ecal_{q,M,\eta}(u_n)$, $(u_n)_n$ is bounded in $\H(B_R)$ and up to a subsequence there exists $u\in \H(B_R)$ such that
	\[ u_n\rightharpoonup u \text{ in }\H(B_R),\quad u_n\to u \text{ in }L^2(B_R),\quad u_n\to u \text{ pointwise a.e. in } B_R. \]
	This implies that \[ \chi_{\left\{u\ne 0\right\}}(x)\leq  \liminf_{n\to+\infty}\chi_{\left\{u_n\ne 0\right\}}(x) \quad \text{ for a.e. } x\in B_R, \]
	so that by Fatou's Lemma there holds 
\begin{equation}\label{eq:lsc measure}
 |\{u\ne 0\}|\leq  \liminf_{n\to+\infty}|\{u_n\ne 0\}|.
\end{equation}
If $|\{u\ne 0\}|<|B_1|$, then \[f_\eta(|\{u\ne 0\}|)= \eta (|\{u\ne 0\}|-|B_1|)\leq\liminf_{n\to+\infty} \eta (|\{u_n\ne 0\}|-|B_1|)\leq \liminf_{n\to+\infty} f_\eta(|\{u_n\ne 0\}|).\]
On the other hand, if $|\{u\ne 0\}|>|B_1|$, then by \eqref{eq:lsc measure} $|\{u_n\ne 0\}|\geq |B_1|$ for $n$ large enough and so
\[f_\eta(|\{u\ne 0\}|)= \frac{1}{\eta} (|\{u\ne 0\}|-|B_1|)\leq \lim_{n\to+\infty} f_\eta(|\{u_n\ne 0\}|).\]
Finally if $|\{u\ne 0\}|=|B_1|$, \[ f_\eta(|\{u\ne 0\}|)=0\leq \liminf_{n\to+\infty} (|\{u_n\ne 0\}|-|B_1|)\leq \liminf_{n\to+\infty} f_\eta (|\{u_n\ne 0\}|).\]
So
\begin{equation}\label{eq:lsc f_eta}
 f_\eta(|\{u\ne 0\}|)\leq \liminf f_\eta(|\{u_n\ne 0\}|).
\end{equation}
Thus by lower semicontinuity with respect to the weak convergence, Proposition \ref{prop:continuity of D(u,u)} and \eqref{eq:lsc f_eta}, $u$ is a minimum of $\Ecal_{q,M,\eta}$. 
Let us show that $\{u\ne 0\}$ has finite perimeter. Define $u_\eps:= (u-\eps)_+-(u+\eps)_-$ for $\eps>0$ small. Since $u$ is a minimum, it holds $\Ecal_{q,M,\eta}(u)\leq  \Ecal_{q,M,\eta}(u_\eps)$, which implies (using also~\eqref{eq:prop feta})
\begin{equation}\label{eq:functional with u_eps}
\int_{\{-\eps<u<\eps\}}|\nabla u|^2\,dx+\eta|\{-\eps<u<\eps\}|\leq -\frac{q}{2}(D(u,u)-D(u_\eps,u_\eps))+ M \Big(\left|\int_{B_R} u_\eps^2\,dx-1\right|-\left|\int_{B_R} u^2\,dx-1\right|\Big).
\end{equation}

Let us analyze each term. Let us start with the right hand side:
\begin{align*}
\left|\int_{B_R} u_\eps^2dx-1\right|-\left|\int_{B_R} u^2dx-1\right|&\leq \left| \int_{\{u>\eps\}} (u-\varepsilon)^2dx+\int_{\{u<-\eps\}} (u+\varepsilon)^2dx -\int_{B_R}u^2dx\right|\\&\leq \left|\int_{\{-\eps<u<\eps\}}u^2 dx\right|+\eps^2|B_R|+4\eps|B_R|^{1/2}||u||_{L^2(B_R)}\\
&\leq 2|B_R|(\eps^2+2\eps||u||_{L^2(B_R)}).
\end{align*}
Then, using also the Hardy-Sobolev inequality \eqref{eq:Hardy-Sobolev}, it holds
\begin{align*}
D(u,u)-D(u_\eps,u_\eps)=&\int_{\{u\ne 0\}}\int_{\{u\ne 0\}}\frac{u(x)u(y)}{|x-y|}\, dxdy-\int_{\{u_\eps\ne 0\}}\int_{\{u_\eps\ne 0\}}\frac{u_\eps(x)u_\eps(y)}{|x-y|}\, dxdy\\
\geq& \int_{\{-\eps<u<\eps\}}\int_{\{-\eps<u<\eps\}}\frac{u(x)u(y)}{|x-y|}\, dxdy+2\int_{\{u<-\eps\}}\int_{\{-\eps<u<\eps\}}\frac{u(x)u(y)}{|x-y|}\, dxdy\\&+2\int_{\{u>\eps\}}\int_{\{-\eps<u<\eps\}}\frac{u(x)u(y)}{|x-y|}\, dxdy+2\int_{\{u<-\eps\}}\int_{\{u>\eps\}}\frac{u(x)u(y)}{|x-y|}\, dxdy\\&-2\int_{\{u<-\eps\}}\int_{\{u>\eps\}}\frac{(u(x)-\eps)(u(y)-\eps)}{|x-y|}\, dxdy
\\\geq &-C(R)\eps^2-C(R)||\nabla u||_{L^1(B_R)}\eps.
\end{align*}
Thus \eqref{eq:functional with u_eps} implies that
\begin{equation*}
\int_{\{-\eps<u<\eps\}}|\nabla u|^2dx+\eta|\{-\eps<u<\eps\}|\leq C(R)(1+||\nabla u||_{L^1(B_R)})\eps.
\end{equation*}
Furthermore using the Cauchy-Schwarz inequality we have
\[ \Big(\int_{\{0<u<\eps\}}|\nabla u|\,dx\Big)^2\leq |\{\{0<u<\eps\}\}|\Big(\int_{\{0<u<\eps\}}|\nabla u|^2\,dx\Big)\leq \frac1\eta (C(R)(1+||\nabla u||_{L^1(B_R)})\eps)^2.  \]
By Coarea formula
\[ \int_{0}^{\eps}P(\{u>s\})ds=\int_{\{0<u<\eps\}}|\nabla u|\, dx \leq \frac{1}{\sqrt{\eta}}C(R)(1+||\nabla u||_{L^1(B_R)})\eps. \]
Finally we find $\delta_n>0$ such that $\delta_n\to 0$ and
\[ P(\{u>\delta_n\})\leq \frac{1}{\sqrt{\eta}}C(R)(1+||\nabla u||_{L^1(B_R)}).\]
Passing to the limit
\[ P(\{u>0\})\leq  \frac{1}{\sqrt{\eta}}C(R)(1+||\nabla u||_{L^1(B_R)}).\]
Similar computations hold for $P(\{u<0\})$ while $P(\{u\ne 0\})$ is bounded as consequence of \[P(\{u\ne 0\})\leq P(\{u<0\})+P(\{u>0\}). \qedhere\]
\end{proof}
\begin{remark}\label{rem:optimal set does not touch the boundary}
In light of Theorem \ref{thm:existenceEq,M,eta}, Lemma \ref{le:surgery}, Theorem~\ref{thm: equivalence unbounded and bounded problem} and Remark \ref{rem: R=2D}, we can suppose there exists an optimal set $\Omega$ of problem \eqref{eq: R q,M,eta in B_R} and it is well separated from $\partial B_R$, in the sense that $\Omega\subseteq B_{R/2}\subseteq B_R$. Moreover any optimal set is connected, otherwise by moving the connected components apart, the functional decreases.
\end{remark}

	We can now show the equivalence between the constrained and the unconstrained problems. {We recall that the constant $M$ has been already fixed, see Remark~\ref{rmk:M}.}
	
	\begin{theorem}\label{thm:noconstraint}
        There exist universal constant $ q_2\in (0,q_1]$ and
          $\eta_2\in (0,\eta_1]$ such that, for all
          $\eta\in (0,\eta_2]$ and $q \in [0, q_2)$, we have that
		\begin{equation}\label{eq:nocontraint}
			\begin{aligned}
				\min\left\{E_{q,M,\eta}(\Om) : \Om\subset B_R \right\} {=}
				\inf\left\{E_q(\Om) : \Omega \subseteq\R^3,\; |\Om|=|B_1|\right\}.
			\end{aligned}
		\end{equation}
		As a consequence, problems \eqref{eq: R q,M,eta in B_R} and \eqref{eq: inf EqOmega} are equivalent {for these values of $q$ and $\eta$}.
	\end{theorem}
    \begin{proof}
        It is easy to check that 	
      \[ 
			\min\left\{E_{q,M,\eta}(\Om) : \Om\subseteq \R^3\right\}
			\le\inf\left\{ E_q(\Om) : \Om\subseteq \R^3,\; |\Om|=|B_1|\right\}
			{=:\mu(q)},
		\]
		as the two functionals coincide on sets of measure $|B_1|$, thanks to
		the definition of $f_\eta$. So by Theorem \ref{thm: equivalence unbounded and bounded problem}, 
              \[ 
			\min\left\{E_{q,M,\eta}(\Om) : \Om\subseteq B_R\right\}
			\le\mu(q),
		\]
        Then, if the reverse inequality
		holds, it follows that on the set of minimizers (of the first or of
		the second problem) the two functionals coincide. 	We prove the claim of the theorem arguing by
		contradiction. Let 
		\[
		\Om_{{q,M,\eta}}\subset B_R,\quad \sigma_{{q,M,\eta}}\in\R,\quad |\Om_{{q,M,\eta}}|=|B_1|+\sigma_{{q,M,\eta}},\quad E_{q,M,\eta}(\Om_{{q,M,\eta}})<\mu(q),
		\]
		and we also note that, $\mu(q)\leq E_q(B_1)$, by definition of infimum.
		We moreover assume, without loss of generality, that $\Om_{{q,M,\eta}}$ are minimizers for problem~\eqref{eq: R q,M,eta in B_R}.
		We treat separately the case $\sigma_{{q,M,\eta}}>0$ and $\sigma_{{q,M,\eta}}<0$. Observe that repeating the computation of Lemma \ref{le:good minizing sequence for surgery} and up to taking $\eta_2$ small enough,  we can suppose that $|\Omega_{q,M,\eta}|$ is uniformly bounded from above and below, thus there exist universal constants $\tilde{C},\tilde c\in(0,1)$ such that $-\tilde c|B_1|\leq \sigma_{q,M,\eta}\leq \tilde C |B_1|$.
        
        {\bf Case $\sigma_{{q,M,\eta}}>0$. }
		Let now ${\rho_{{q,M,\eta}}}<1$ be such that
		$|{\rho_{{q,M,\eta}}}\Om_{{q,M,\eta}}|=|B_1|$, so that 
		\[
		{ {\rho_{{q,M,\eta}}}= 1
			-\frac{\sigma_{{q,M,\eta}}}{3|B_1|}+C\sigma_{{q,M,\eta}}^2,}
		\]
		for some {$C = C(\sigma_{q,M,\eta}) \in \R$ such   that $|C| \leq C_0$ for all  $|\sigma_{q,M,\eta}| <\tilde C |B_1|$ some $C_0 > 0$ universal}.
		
		We call $u=u_{q,M,\eta}$ an optimal normalized function attaining $E_q(\Omega_{q,M,\eta})$, thus the function\[
		\widetilde u(y)={\rho_{{q,M,\eta}}}^{-\frac{3}{2}}u\Big(\frac{y}{{\rho_{{q,M,\eta}}}}\Big),\qquad y\in {\rho_{{q,M,\eta}}}\Omega_{q,M,\eta},
		\]
		is an admissible  competitor with unitary $L^2-$norm for $E_q({\rho_{{q,M,\eta}}}\Om_{q,M,\eta})$.
		We  have the following scalings
		\begin{equation*}\label{eq:split}
	\begin{split}
	    			\int_{{\rho_{{q,M,\eta}}}\Omega_{q,M,\eta}} |\nabla \widetilde u(y)|^2\,dy&={\rho_{{q,M,\eta}}}^{-2}\int_{\Om_{q,M,\eta}}|\nabla u(x)|^2\,dx,\\
                		D(\widetilde u,\widetilde u)
				&={{\rho_{{q,M,\eta}}}^{2}}D(u,u).
	\end{split}    
		\end{equation*}

		Since the new set $\rho_{{q,M,\eta}}\Om_{{q,M,\eta}}$ is now admissible in the constrained minimization problem~\eqref{eq: inf EqOmega}, using the  above scaling we obtain
        	\[
		\begin{split}
			E_{q,M,\eta}(\Om_{q,M,\eta})&=E_q(\Om_{{q,M,\eta}})+\frac{\sigma_{{q,M,\eta}}}{\eta}
			<\mu\leq E_q({\rho_{{q,M,\eta}}}\Om_{{q,M,\eta}})\\
			&\leq \int_{{\rho_{{q,M,\eta}}}\Omega_{q,M,\eta}}|\nabla \widetilde u(y)|^2\,dy+\frac{q}{2}D(\widetilde u,\widetilde u)\\&={\rho_{{q,M,\eta}}}^{-2}\int_{\Om_{q,M,\eta}}|\nabla u(x)|^2\,dx+{\rho_{{q,M,\eta}}}^{2}\frac{q}{2} D(u,u)\\
			&=\int_{\Om_{q,M,\eta}}|\nabla u(x)|^2\,dx \left(1+{\frac{2\sigma_{{q,M,\eta}}}{3|B_1|}}+C\sigma_{{q,M,\eta}}^2\right)
			+\frac{q}{2}  D(u,u)\left(1-{\frac{2\sigma_{{q,M,\eta}}}{3|B_1|}}+C\sigma_{{q,M,\eta}}^2\right),
		\end{split}
		\]
		we deduce that (up to increasing $C$, recalling also that
		$E_q(\Omega_{q,M,\eta})\leq E_q(B_1)$)
		\[
		\begin{aligned}
			\frac{\sigma_{{q,M,\eta}}}{\eta}&<\int_{\Om_{q,M,\eta}}|\nabla u(x)|^2\,dx \left({\frac{2\sigma_{{q,M,\eta}}}{3|B_1|}}\right)-\frac{q}{2}  D(u,u)\left({\frac{2\sigma_{{q,M,\eta}}}{3|B_1|}}\right)+ 2E_q(B_1)C\sigma_{{q,M,\eta}}^2\\
			&\le {\frac{\sigma_{q,M,\eta}}{3|B_1|}} 2 E_q(\Omega_{q,M,\eta})+C\sigma_{{q,M,\eta}}^2.
		\end{aligned}
		\]
		Thus, for some universal $C>0$,
		\[
		\frac{1}{\eta}\le C E_q(\Omega_{q,M,\eta})+C\sigma \le C E_q(B_1)\leq C E_1(B_1),
		\]
		which leads to a contradiction as soon as $\eta_2<\frac{1}{C \,E_1(B_1)}.$

        {\bf Case $\sigma_{{q,M,\eta}}<0$. }  Let now ${\rho_{{q,M,\eta}}}>1$ be such that
		$|{\rho_{{q,M,\eta}}}\Om_{{q,M,\eta}}|=|B_1|$ and consider the function $g\colon[1,\rho_{{q,M,\eta}}]\rightarrow \R$ defined by
        \[
		 g(r)=\int_{r\Om_{q,M,\eta}}|\nabla  u_r|^2\, dx+\frac{q}{2} \int_{r\Om_{{q,M,\eta}}}\int_{r\Omega_{q,M,\eta}}\frac{ u_r(x) u_r(y)}{|x-y|}\,dxdy+\eta(r^3|\Om_{{q,M,\eta}}|-|B_1|),
		\]
        where
\[
		 u_r(y)={r}^{-\frac{3}{2}}u\Big(\frac{y}{r}\Big),\qquad y\in r\Omega_{q,M,\eta},
		\]We show that the minimum of the function $g$ is attained at $r=\rho:=\rho_{{q,M,\eta}}$. 
Then the proof is concluded by Proposition \ref{prop:equivalence E_q and E_q,M} because this implies that $E_q(\rho_{q,M,\eta}\Omega_{q,M,\eta})=E_{q,M,\eta}(\rho_{q,M,\eta}\Omega_{q,M,\eta})\leq E_{q,M,\eta}(\Omega_{q,M,\eta})<\mu(q)$ which leads to a contradiction.
Proving that $g$ has a minimum at $r=\rho$ is equivalent to show that for some $\eta$ the inequality  
		\[
		g(r)\geq \int_{\rho\Om_{q,M,\eta}}|\nabla \widetilde u|^2\, dx+\frac{q}{2} \int_{\rho\Om_{{q,M,\eta}}}\int_{\rho\Omega_{q,M,\eta}}\frac{\widetilde u(x)\widetilde u(y)}{|x-y|}\,dxdy,\qquad \text{for all }r\in[1,\rho],
		\]
		holds true, where $\widetilde u$ is defined as before. Up to rearranging the terms, and by the rescaling of the involved integrals, such an inequality reads as
		\[
		\eta\left(1-\left(\frac{r}{\rho}\right)^3\right)\le \int_{\rho\Om_{q,M,\eta}}|\nabla \widetilde u|^2\left(\left(\frac{r}{\rho}\right)^{-2}-1\right)+\frac{q}{2} D(\widetilde u,\widetilde u)\left(\left(\frac{r}{\rho}\right)^{2}-1\right).
		\]
		Setting $t:=\frac r\rho < 1$, and observing that $r^3|\Om_{{q,M,\eta}}|=t^3$, the last inequality is equivalent to
		\[
		\eta\le \frac{\int_{\rho\Om_{q,M,\eta}}|\nabla \widetilde u|^2\, dx(t^{-2}-1)-\frac{q}{2} D(\widetilde u,\widetilde u)(1-t^2)}{1-t^3}.
		\]
        Moreover by scaling and $|\sigma_{q,M,\eta}|\geq -\tilde c |B_1|$, it holds
        \[
       		D(\widetilde u,\widetilde u)={{\rho_{{q,M,\eta}}}^{2}}D(u,u)\leq (1-\tilde c)^{-2/3}E_1(B_1)=\overline{C}
        \]
        where $\overline{C}$ is a universal constant.
It is easy to check that the right hand side is bounded from below by the function \[
		t\mapsto\frac{\lambda_0(B_2)(t^{-2}-1)-\frac{q}{2}\overline{C}(1-t^2)}{1-t^3},\qquad t\in(0,1),
		\]
		which is a function strictly decreasing in $(0,1)$ for $q\leq 2/3\overline{C}$ and with infimum given by \[
		\lim_{t\to 1^-}\frac{\lambda_0(B_2)(t^{-2}-1)-q\overline{C}(1-t^2)}{1-t^3}=\frac{2\lambda_0(B_2)-q\overline{C}}{3}>0,
		\]
		as $q\leq q_2$ where $q_2=\frac{2}{\overline{C}}\min\{ \lambda_0(B_2),1/3\}.$
		Thus it is enough to take $\eta\leq \eta_2\leq \frac{2\lambda_0(B_2)-q_2\overline{C}}{3}$ and we immediately deduce that $g$ has minimum for $r=\rho$.
		This concludes the proof.
    \end{proof}

    \begin{remark}\label{rmk:etafissato}
       { Hereafter we fix $\eta=\frac{\eta_2}{2}$. In light of Theorem \ref{thm:noconstraint}, we know that all optimizers $\Omega$ of problem \eqref{eq: R q,M,eta in B_R} satisfy $|\Omega|=|B_1|$.
       }
    \end{remark}

\subsection{Non-negativity of the optimal function}
Before delving into the regularity of the optimal sets, it is crucial to determine whether the optimal function provided by Theorem \ref{thm:existenceEq,M,eta} is nonnegative or sign-changing. This distinction is fundamental: if the function is nonnegative, the problem can be treated using one-phase free boundary techniques; otherwise, a two-phase free boundary argument is required. The idea is to obtain the nonnegativity of our optimal function by testing its minimality against its positive part. To do so, it is useful to know some other information about the optimal function, such as some convergences as $q$ goes to zero.

\begin{proposition}\label{prop: Gamma convergence} Let $q\in(0,q_2)$. Then  $\Ecal_{q,M,\eta}$ $\Gamma$-converges to $\Ecal_{0,M,\eta}$ as $q\to 0$ with respect to the weak topology of $\H(B_R)$ and 
	\begin{equation*}
		\inf\left\{\Ecal_{q,M,\eta}(u):u\in \H(B_R)\right\}\to \lambda_0(B_1).
	\end{equation*}
 Moreover for all  $(u_q)_q\subseteq \H(B_R)$ such that each $u_q$ minimizes $\Ecal_{q,M,\eta}$, then every weak limit of a subsequence in $\H(B_R)$ is either $u_{B_1}$ or $-u_{B_1}$, denoting $u_{B_1}$ is the first positive normalized eigenfunction of the Dirichlet Laplacian on $B_1$.
 \end{proposition}
	\begin{proof}
		Let $(v_q)_q\subseteq \H(B_R)$ be a sequence weakly converging to $v\in \H(B_R)$ in $\H(B_R)$ as $q\to 0$. Then $(v_q)_q$ is bounded in $\H(B_R)$ and so, up to a subsequence, $(v_q)_q$ converges pointwise and in $L^2(B_R)$ to $v$ as $q\to 0$. Since $|v\ne 0|\leq \liminf\limits_{q\to 0}|\left\{v_q\ne 0\right\}|$ and by Proposition \ref{prop:continuity of D(u,u)}, it holds
		\[ \Ecal_{0,M,\eta}(v)\leq \liminf_{q\to 0}\Ecal_{q,M,\eta}(v_q),\]
		so the \textit{liminf inequality} holds. Instead the \textit{limsup inequality} holds by taking the sequence $v_q=v$ for all $q\in(0,1)$. Thus  $\Ecal_{q,M,\eta}$ $\Gamma$-converges to $\Ecal_{0,M,\eta}$ as $q\to 0$ with respect to the weak $\H(B_R)$ topology. By Theorem \ref{thm:existenceEq,M,eta}, for all $q\in(0,1)$, there exists $u_q\in H^1_0(B_R)$ such that 
		\[ 	\inf\left\{\Ecal_{q,M,\eta}(u):u\in \H(B_R)\right\} = \mathcal{E}_{q,M,\eta}(u_q).\]
		Since for all $q\in(0,1)$ it holds (recalling also~Lemma~\ref{le:positivityD(u,u)} and the fact that $|\{u_q\not =0\}|=|B_1|$)
		\[ 
        \int_{B_R}|\nabla u_q|^2dx\leq \mathcal{E}_{q,M,\eta}(u_q)\leq E_1(B_1)+|B_1|,
        \]
		so $(u_q)_q$ is bounded in $\H(B_R)$ and by reflexivity of $\H(B_R)$, there exists a subsequence which weakly converging in $\H(B_R)$. Then by a well-known property of the $\Gamma$-convergence (see for example~\cite[Chapter 7]{DalMaso}), the minimum and the minimizers of $\Ecal_{q,M,\eta}$ converge respectively to the minimum and the minimizers of $\Ecal_{0,M,\eta}$, respectively. 
        By Theorem \ref{thm:noconstraint} and the Faber-Krahn inequality,\[\inf\limits_{u\in\H(B_R)}\mathcal{E}_{0,M,\eta}(u)=\lambda_0(B_1).\] Since $u_{B_1}$ and $-u_{B_1}$ are the only minimizers of $\Ecal_{0,M,\eta}$ on $\H(B_R)$, where $u_{B_1}$ is the first positive normalized eigenfunction of the Dirichlet Laplacian, the proof is concluded.
	\end{proof}

Despite Proposition \ref{prop: Gamma convergence} gives a lot of information about some convergences of the optimal function, we need a stronger (uniform) one. In the following proposition we prove uniform convergence of minimizers as a consequence of the following regularity result, inspired by the regularity theory introduced by Giaquinta and Giusti in \cite{GiaquintaGiusti}. 
\begin{proposition}\label{prop: L^infty convergence}
 Let $q\in(0,q_2)$. Then for all $(u_q)_q\subseteq \H(B_R)$ such that $u_q$ minimizes $\Ecal_{q,M,\eta}$, up to possibly taking $-u_q$ for some $q$, $(u_q)_q$ uniformly converges to $u_{B_1}$ in $B_R$ as $q\to 0$.
 \end{proposition}
\begin{proof}
    Let $u\in \H(B_R)$ be an optimal function of problem \eqref{eq:min mathcal E q,M,eta}. Let $\widehat{x}\in B_{R-1}$ and $0<r<1$. Let us consider a function $\varphi\in \H(B_r(\widehat x))$ and define $v=u+\phi$. Then by minimality it holds
    \begin{align}\label{eq:min in holder}
        \int_{B_R}|\nabla u|^2+\frac{q}{2}D(u,u)+&M\Big| \int_{B_R}u^2dx-1 \Big|+f_\eta(|\{u\ne 0\}|)  \\ &\leq    \int_{B_R}|\nabla v|^2+\frac{q}{2}D(v,v)+M\Big| \int_{B_R}v^2dx-1 \Big|+f_\eta(|\{v\ne 0\}|).   
    \end{align}
    By \eqref{eq:prop feta}, $\bigg|f_\eta(|\{v\ne 0\}|)-f_\eta(|\{u\ne 0\}|)\bigg|\leq C(\eta)r^3$ where $C(\eta)=C\min \{\eta,1/\eta\}$ and $C$ is a universal constant possibly increasing from line to line. Considering the $L^2$ term, we obtain, for a universal constant $C>0$, 
    \[ \Bigg| \Big| \int_{B_R}u^2dx-1 \Big|-\Big| \int_{B_R}v^2dx-1 \Big| \Bigg|\leq \Bigg| \int_{B_r(\widehat x)} u^2-v^2 dx\Bigg|\leq C (||u||^2_{L^\infty(B_R)}+||v||^2_{L^\infty(B_R)})r^3.\]
    Finally the non local term gives
    \begin{equation}
    \begin{split}
        D(v,v)-D(u,u)&=D(u+\phi,u+\phi)-D(u,u)= 2D(u,\phi)+D(\phi,\phi)\\
        &=2\int_{B_r(\widehat x)}v_u \varphi dx+\int_{B_r(\widehat x)}\int_{B_r(\widehat x)} \frac{\phi(x)\phi(y)}{|x-y|}\,dxdy.
    \end{split}
    \end{equation}
    Since 
    \begin{align}
        \int_{B_r(\widehat x)}v_u \varphi dx\leq &C \|v_u\|_{L^\infty(B_R)}\|\varphi\|_{L^\infty(B_R)}r^3
    \end{align}
    and by \eqref{eq:Hardy-Sobolev} and \cite[Lemma 2.4]{FuscoPratelli}
    \begin{align}
        D(\phi,\phi)\leq C \|\phi\|^2_{L^\infty(B_R)}r^3.
    \end{align}
    Thus \eqref{eq:min in holder} becomes
        \begin{align}\label{eq:min in holder 2}\small
        \int_{B_r(\widehat x)}|\nabla u|^2\leq \int_{B_r(\widehat x)}|\nabla v|^2+C\Bigg((||u||^2_{L^\infty(B_R)}+||v||^2_{L^\infty(B_R)})+C(\eta)\Bigg)r^3.
    \end{align}
    for a universal constant $C$.
    Let us now consider a suitable explicit test function $v$, namely the solution to
\[ 
    \begin{cases}
        \Delta v=0 & \text{in }B_r(\widehat x)\\
        v=u & \text{on } B^c_r(\widehat x)\,.
    \end{cases}\]
    Observe that $\int_{B_r(\widehat x)} |\nabla v|^2dx\leq \int_{B_r(\widehat x)} |\nabla u|^2dx$ and $||v||_{L^\infty(B_R)}\leq ||u||_{L^\infty(B_R)}$ by weak maximum principle (see \cite[Theorem 8.1]{GilbargTrudinger}). Moreover $u-v\in \H(B_r(\widehat x))$, so $v$ can be taken as test function in \eqref{eq:min in holder}. 
    Then \eqref{eq:min in holder 2} becomes, recalling that $u$ is uniformly bounded in $L^\infty$, see Theorem~\ref{thm: pde and L infty estimate},
\begin{equation}\label{eq:stima per Campanato}
         \int_{B_r(\widehat x)}|\nabla u|^2\leq \int_{B_r(\widehat x)}|\nabla v|^2 + C(\eta)r^3.  
\end{equation}
By the equation solved by $v$, it holds $\int_{B_r(\widehat x)}\nabla v\cdot \nabla (u-v)dx=0$, so $\int_{B_r(\widehat x)}|\nabla v|^2dx=\int_{B_r(\widehat x)}\nabla v \cdot \nabla u\, dx$. Finally we deduce 
\[ \int_{B_r(\widehat x)}|\nabla (u-v)|^2dx=\int_{B_r(\widehat x)}|\nabla v|^2dx+\int_{B_r(\widehat x)}|\nabla u|^2dx-2\int_{B_r(\widehat x)}\nabla u\cdot \nabla vdx= \int_{B_r(\widehat x)}|\nabla u|^2dx-\int_{B_r(\widehat x)}|\nabla v|^2dx.\]
Thus by \eqref{eq:stima per Campanato} we obtain
\begin{align}\label{eq:gradienti funzioni armoniche}
   \int_{B_r(\widehat x)}|\nabla (u-v)|^2dx\leq  C(\eta) r^{3} .
\end{align}
Consider $0<\rho<r$. Since $v$ is harmonic, then also $\partial_iv$ for $i=1,2,3$ is harmonic. So by the mean value theorem for harmonic functions and Jensen's inequality, there holds
\[ |\nabla v(x)|^2=\sum_{i=1}^3 |\partial_iv|^2=\sum_{i=1}^3 \Big(\mean{B_r(\widehat x)}\partial_iv\,dy\Big)^2\leq \sum_{i=1}^3 \Big(\mean{B_r(\widehat x)}|\partial_iv|^2\,dy\Big) =\mean{B_r(\widehat x)}|\nabla v|^2\,dx. \]
Then, integrating over $B_\rho(\widehat x)$ the above chain of inequalities, we obtain
\begin{equation}\label{eq: riscalamento armoniche}
     \int_{B_\rho(\widehat x)}|\nabla v|^2\,dx\leq\frac{|B_\rho|}{|B_r|}\int_{B_r(\widehat x)}|\nabla v|^2\,dx. 
\end{equation}
Thus by \eqref{eq:gradienti funzioni armoniche} and \eqref{eq: riscalamento armoniche}, there holds
\begin{align}\label{eq: iteration lemma 1}
    \int_{B_\rho(\widehat x)}|\nabla u|^2dx= &  \int_{B_\rho(\widehat x)}|\nabla \big((u-v)+v\big)|^2dx\le \int_{B_r(\widehat x)}|\nabla (u-v)|^2dx+  \int_{B_\rho(\widehat x)}|\nabla v|^2dx\\ \leq& C(\eta)r^3+\Big(\frac{\rho}{r}\Big)^3\int_{B_r(\widehat x)}|\nabla v|^2dx\leq Cr^3+\Big(\frac{\rho}{r}\Big)^3\int_{B_r(\widehat x)}|\nabla u|^2 dx.
\end{align}
Set $\psi(s):=\int_{B_s(\widehat x)}|\nabla u|^2dx$, then \eqref{eq: iteration lemma 1} can be rewritten as
\begin{equation}
    \psi(\rho)\leq C(\eta)r^{3}+\Big(\frac{\rho}{r}\Big)^3\psi(r).
\end{equation}
In light of \cite[Lemma 2.1, Chapter 3]{Giaquinta} or~\cite[Lemma~5.13]{giaqmart}, then for all $\beta\in (0,3)$, it holds 

\begin{equation}\label{eq:gradient estimate after Giaquinta lemma}
    \int_{B_\rho(\widehat x)}|\nabla u|^2= \psi(\rho)\le C(\eta, \beta)\rho^{\beta}. 
\end{equation}
where $C(\eta,\beta)$ is a universal constant depending only on $\eta, \beta$.  Let us choose $\beta=2$. Finally by the classical Campanato criterion (for further details see \cite[Chapter 3]{Giaquinta} or~\cite[Section~5]{giaqmart}),
\begin{equation}\label{eq: holder estimate local}
    \|u\|_{C^{0,1/2}(B_r(\widehat x))}\leq C(\eta),
\end{equation}
where $C=C(\eta)$ is a constant independent of $q,\,r$ and $\widehat x$. Then we can consider a finite covering of $B_{R/2}$ with balls $B_r(x_i)$ for some $x_i\in B_{R-1}$. Since by Remark \ref{rem:optimal set does not touch the boundary}, $\{u\ne 0\}\subseteq B_{R/2}\subseteq B_R$, then we can cover $\{u\ne 0\}$ with a finite number of balls of radius $r$, to obtain 
\begin{equation}\label{eq:holder estimate final}
    \|u\|_{C^{0,1/2}(B_{R/2})}\leq C(\eta),
\end{equation}
where $C=C(\eta)$. Finally by Proposition \ref{prop: Gamma convergence} there exists a sequence of functions $(u_q)$ minimizing $\mathcal{E}_{q,M,\eta}$ that weakly converge as $q\to 0$ to $u_{B_1}$ (or $-u_{B_1}$), the first positive normalized eigenfunction of the Dirichlet Laplacian on $B_1$. For each element of this sequence, \eqref{eq:holder estimate final} holds. Then by Ascoli-Arzel\`a, up to a subsequence, there exists $v\in \H(B_R)$ such that $u_q$ converge uniformly to $v$.  By uniqueness of the limit $v=u_{B_1}$ (or $-u_{B_1}$). Since $-u_q$ is still a minimizer for $\mathcal{E}_{q,M,\eta}$, then up to changing some elements in $(u_q)_q$ with its opposite, without loss of generality we can suppose that the uniform limit is $u_{B_1}.$
\end{proof}
\begin{remark}\label{rem:d dimensional version of holder estimate}
Observe that a local H\"older estimate follows also by the $C^{\infty}(\Omega_q)$ regularity proven in Theorem \ref{thm: pde and L infty estimate}. The difference is that \eqref{eq: holder estimate local} holds locally in $B_R$ and not only in $\Omega_q$, allowing to obtain a global estimate.  Moreover dealing with the $d$-dimensional version of this problem, Proposition \ref{prop: L^infty convergence} still holds taking $\beta=d-1$ in \eqref{eq:gradient estimate after Giaquinta lemma}.
\end{remark}
 \noindent
Thanks to the aforementioned convergences, now we can prove the nonnegativity of an optimal function.
\begin{proposition}\label{prop: positivity of the minimizers}
	There exists $q_3=q_3(\eta)$ such that for all $q\leq q_3$ all optimal functions $u_q\in \H(B_R)$ attaining the infimum of $\Ecal_{q,M,\eta}$ on $\H(B_R)$ have constant sign.
	\begin{proof}
		By Theorem \ref{thm:existenceEq,M,eta}, we know that there exists a minimizer $u\in \H(B_R)$ for $\Ecal_{q,M,\eta}$ on $\H(B_R)$. Let us show that testing the minimality of $u$ against $u^+$ leads to the claim. By minimality it holds
		\begin{align*}
			\int_{B_R} |\nabla u(x)|^2\,& dx+\frac{q}{2}\int_{B_R} \int_{B_R} \frac{u(x)u(y)}{|x-y|}\, dxdy+ M\left|\int_{B_R}u^2dx-1\right|+f_\eta(|\{u\ne 0\}|)\\ \leq& 	\int_{B_R} |\nabla u^+(x)|^2\, dx+\frac{q}{2}\int_{B_R} \int_{B_R} \frac{u^+(x)u^+(y)}{|x-y|}\, dxdy+ M\left|\int_{B_R}(u^+)^2dx-1\right|+f_\eta(|\{u^+\ne 0\}|),
		\end{align*}
		Moreover 
	\begin{align*}
		\left|\int_{B_R} u^2dx-1\right|-\left|\int_{B_R} (u^+)^2dx-1\right|\leq \left| \int_{B_R} u^2dx - \int_{B_R} (u^+)^2dx\right|= \int_{\{u<0\}} u^2dx,
	\end{align*} 
	and by \eqref{eq:prop feta}
	\[f_\eta(|\{u\ne 0\}|)-f_\eta(|\{u^+\ne 0\}|)\geq \eta |\{u<0\}|. \]
	So
			\begin{align*}
			\int_{\{u<0\}} |\nabla u(x)|^2&\, dx+\frac{q}{2}\int_{\{u<0\}} \int_{\{u<0\}} \frac{u(x)u(y)}{|x-y|}\, dxdy\\&+q\int_{\{u>0\}} \int_{\{u<0\}} \frac{u(x)u(y)}{|x-y|}\, dxdy -  M \int_{\{u<0\}}u^2dx+\eta |\{u<0\}|\leq 0.
		\end{align*}
	By Hardy-Sobolev  and H\"older inequalities there holds
	\begin{align*}
		\left|
		\int_{\{u>0\}} \int_{\{u<0\}} \frac{u(x)u(y)}{|x-y|} dxdy\right|\leq&  \int_{\{u>0\}} \int_{\{u<0\}} \frac{|u(x)||u(y)|}{|x-y|} dxdy \\ 
		\leq& C||\nabla u||_{L^1(B_R)} \int_{\{u<0\}} |u(x)|dx \\ 
		\leq& C||\nabla u||_{L^1(B_R)}||u^-||_{L^\infty(B_R)} |\{u<0\}|\\
		\leq& C(R)||u^-||_{L^\infty(B_R)}||\nabla u||_{L^2(B_R)} |\{u<0\}|.
	\end{align*}
	Since $u$ minimizes $\Ecal_{q,M,\eta}$, then $||\nabla u||_{L^2(B_R)}\leq \Ecal_{q,M,\eta}(u)\leq E_1(B_1)+|B_1|$, implying
	\begin{equation}\label{eq: |{u<0}| inequality}
		-q\int_{\{u>0\}} \int_{\{u<0\}} \frac{u(x)u(y)}{|x-y|} dxdy\geq - C(R)||u^-||_{L^\infty(B_R)}|\{u<0\}|.
	\end{equation}
	So it holds
	\[	\int_{\{u<0\}} |\nabla u(x)|^2\, dx+\frac{q}{2}\int_{\{u<0\}} \int_{\{u<0\}} \frac{u(x)u(y)}{|x-y|}\, dxdy+ \Big(\eta-M||u^-||^2_{L^\infty(B_R)}- C(R)||u^-||_{L^\infty(B_R)}\Big) |{u<0}|\leq 0.\]
	Finally, by Proposition \ref{prop: L^infty convergence} and up to taking $-u$, there exists $q_3(\eta)$ such that for all $q\in(0,q_3)$ it holds
	\[\eta-M||u^-||^2_{L^\infty(B_R)}- C(R)||u^-||_{L^\infty(B_R)}>0.\]
		Thus by \eqref{eq: |{u<0}| inequality}, $|\{u<0\}|=0$.
	\end{proof}
    \begin{remark}
        Since  $u_q$ and $-u_q$ are minimizers of $\mathcal{E}_{q,M,\eta}$, by Proposition \ref{prop: positivity of the minimizers} and without loss of generality, we can consider the optimal functions to be nonnegative. 
    \end{remark}

\end{proposition}
\subsection{Free boundary formulation}\label{sec:freeboundary}
In this section we want to improve the regularity for our optimal set. 
	\begin{lemma}\label{le:non degeneracy}
          Let  $q\in (0,q_3]$, let $\Om$ be an  optimal set for problem~\eqref{eq: R q,M,eta in B_R},
          and let $u\in H^1_0(\Omega)$ be any  (nonnegative) function attaining  $E_q(\Omega)=E_{q,M}(\Omega)$.
          Then for every $\kappa\in (0,1)$ there are positive constants $K_0,\rho_0$ depending only on $\kappa, \eta$  such that the following assertion holds: if  $\rho\leq \rho_0$ and $x_0\in \overline{B_R}$, then
		\begin{equation}\label{eq:nondegmean}
			\mean{\partial B_\rho(x_0)\cap B_R}{u\,d\mathcal H^{2}}\leq K_0 \rho\,\,\,\,
			\Longrightarrow \,\,\,\,u\equiv0 \text{ in }\,\,B_{\kappa\rho}(x_0)\cap B_R.
		\end{equation}
	\end{lemma}
    \begin{proof}
        The proof is the same of \cite[Lemma 3.9]{MazzoleniMuratovRuffini}. Let us just point out that 		
        \[
		\gamma_1:=2\sup_x\left(\lambda_q u(x)-q v_u(x) \right)>0. 
		\]
        Indeed suppose by contradiction that $\gamma_1\le0$. Then $u(x)\leq q \frac{v_u(x)}{\lambda_q}$ for all $x\in \R^3.$ Recalling that $\|v_u(x)\|_{L^\infty(\R^3)}\leq C E_1(B_1)$ by \eqref{eq:Hardy-Sobolev} where $C$ is a universal constant, and $\lambda_q\geq \lambda_0(\Omega)\geq \lambda_0(B_1)$ since $|\Omega|\leq|B_1|$. Then by nonnegativity of $u$, 
        \[ ||u||_{L^\infty(\R^3)}\leq C q.\]
        Finally $u$ has unitary $L^2$ norm and $|\Omega|=|B_1|$, so we obtain a contradiction as $q$ is small enough.
    \end{proof}

		\begin{lemma}\label{le:lipschitz}
		Let $\eta$, $q$, $\Om$ and $u$ be as in	Lemma~\ref{le:non degeneracy}. The function $u$ can be extended to a Lipschitz continuous function defined in the whole $B_R$, with Lipschitz constant $L=L(\eta)$. In particular,	$\Omega=\{u>0\}\subset B_R$ is an open set.
	\end{lemma}
\begin{proof}
    We follow the approach of~\cite[Section~3.2]{VelichkovLecture}, first proposed in~\cite{BrianconHayouniPierre}.
    
    \noindent{\it Step 1.} We prove an estimate on the nonnegative Radon measure $|\Delta u|$, namely
		\begin{equation}\label{eq:laplest}
                  |\Delta u|(B_r{(x_0)})\leq C
                  r^2,\qquad{\text{for all $x_0\in B_R$ and
                      {$0 < r < 1$} such that $B_{2r}(x_0)\subset B_R$}}
		\end{equation}
        for a universal constant $C>0$.  Let $\psi\in C^\infty_c({B_{2r}(x_0)})$ for some ${B_{2r}(x_0)}\subset B_R$ , with $\|\psi\|_{L^\infty}\leq c$, and we test the optimality of $u$ against $u+\psi$, obtaining:\[
		\int_{B_R} |\nabla u|^2\, dx+\frac{q}{2}D(u,u)+f_\eta(|\{u>0\}|)\leq \int_{B_R} |\nabla (u+\psi)|^2\, dx+\frac{q}{2}D\Big(u+\psi,u+\psi\Big)+f_\eta(|\{u+\psi\ne0\}|),
		\]
		which implies 
		\[
		-{2} \int_{{B_{2r}(x_0)}}\nabla u\cdot \nabla \psi\, dx\leq
		\int_{{B_{2r}(x_0)}}|\nabla \psi|^2\, dx+C_\eta|\{u=0\}\cap
		{B_{2r}(x_0)}|+\frac{q}{2}\int_{{B_{2r}(x_0)}}\int_{{B_{R}(0)}}P(x,y)\,dxdy
		\]
		where 
		\[
                  P(x,y)=\frac{\psi(x)\psi(y)+2u(x)\psi(y)}{|x-y|}
		\]
        Recalling that $\|\psi\|_{L^\infty}\leq c$ and using Theorem~\ref{thm: pde and L infty estimate}, we can
                  control the nonlocal term as
                  \[
                    \int_{{B_{2r}(x_0)}}\int_{{B_{R}(0)}}P(x,y)\,dxdy\leq
                    C_1
                    \int_{{B_{2r}(x_0)}}\int_{{B_{R}(0)}}\frac{1}{|x-y|}\,dxdy\leq
                    C_2 {R^2} |B_{2r}(x_0)|\leq C_3r^3.
\] 
	Thus we obtain 
	\begin{equation}\label{eq:estgrad}
			\begin{split}
                          &- {2} \int_{{B_{2r}(x_0)}}\nabla u\cdot
                          \nabla \psi\, dx\leq
                          \int_{{B_{2r}(x_0)}}|\nabla \psi|^2\,
                          dx+C_\eta|\{u=0\}\cap
                          {B_{2r}(x_0)}|+{Cqr^3}.
			\end{split}
		\end{equation}
        	We now set, for all $\varphi\in     C^\infty_c({B_{2r}(x_0)})$,  $\psi={\pm} r^{3/2}\|\nabla  \varphi\|_{L^2}^{-1}\varphi$   and from~\eqref{eq:estgrad} we deduce, for some $\widetilde C>0$
		\[
		\Big|\int_{{B_{2r}(x_0)}}\nabla u\cdot \nabla \varphi\, dx\Big|\leq {\widetilde C} r^{3/2}\|\nabla \varphi\|_{L^2({B_{2r}(x_0)})}
		\]
		It is then enough to choose
                $\varphi\in C^\infty_c(B_{2r}{(x_0)})$ with
                $\varphi\geq 0$ and $\varphi=1$ in ${B_{r}(x_0)}$
                and with
                $\|\nabla \varphi\|_{L^\infty(B_{2r})}\leq \frac2r$
                (notice that this is compatible with the requirement
                $\|\psi\|_{L^\infty}\leq c$ {independently of
                  $r$}) to obtain, for some constant $C>0$:
		\begin{equation}\label{eq:estlapl}
                  |\Delta u|({B_r(x_0)})\leq|\Delta u|(\varphi)=
                  \left|\int_{{B_{2r}(x_0)}}\nabla u\cdot\nabla
                    \varphi\,dx\right| \leq C r^2. 
		\end{equation}

	\noindent{\it Step 2.} We prove that the Laplacian estimate~\eqref{eq:estlapl} of Step 1
		entails (recall that $\mathcal H^2(\partial B_r)=4\pi {r^2}$)
		\begin{equation}\label{eq:intest}
                  \frac{1}{4\pi {r^2}}\int_{\partial
                    {B_r(x_0})}u\,d\mathcal H^2\leq {u(x_0)}+ Cr\qquad
                  {\text{for all }x_0\in B_R}, 
		\end{equation}
		for some constant $C>0$. This follows from~\cite[Lemma~3.6]{BrianconHayouniPierre}, which assures that, for all $x_0\in B_R$, it holds
			\begin{equation}\label{eq:BHP3.6}
				\frac{1}{4\pi r^2}\int_{\partial B_r(x_0)}u\,d\mathcal H^2-u(x_0)=\int_0^r\frac{1}{4\pi s^2}\Delta u(B_s(x_0))\,ds.
			\end{equation} 
			It is then enough to put together~\eqref{eq:BHP3.6} and~\eqref{eq:estlapl} to obtain~\eqref{eq:intest}.
Now, let us take $x_0\in \partial \{u>0\}\cap B_R$ and a sequence of $x_n\to x_0$ such that $u(x_n)=0$ for all $n$ and with $x_n\in B_{r_1}(x_0)\subset B_R$ for some $r_1>0$. For those points~\eqref{eq:intest} reads as
\begin{equation}\label{eq:xn}
  \frac{1}{4\pi {r^2}}\int_{\partial
                    {B_r(x_n})}u\,d\mathcal H^2\leq u(x_n)+Cr=Cr,\qquad \text{for all }r< r_1,
\end{equation}
and the constant $C$ does not depend on $n$. 
Since $u\in H^1(B_R)$, the map $x\mapsto \tfrac{1}{4\pi {r^2}}\int_{\partial
                    {B_r(x})}u\,d\mathcal H^2$ is continuous, see~\cite[Remark~3.6]{VelichkovLecture}.
We can then pass to the limit as $n\to\infty$ in~\eqref{eq:xn} to deduce \[
\frac{1}{4\pi {r^2}}\int_{\partial
                    {B_r(x_0})}u\,d\mathcal H^2\leq Cr,\qquad \text{for all }r< r_1.
\]		
Finally, passing to the limit  as $r\to 0$, we obtain that $u(x_0)=0$ (recalling that we are considering the quasi continuous representative of the Sobolev function $u$), thus
                  $\Omega\cap \partial\Omega=\{u>0\}\cap
                  \partial\{u>0\}=\emptyset$, hence $\Omega=\{u>0\}$
                  is an open set. 
                  
       \noindent{\it Step 3.} By previous steps and \cite[Lemma 2.4]{GustafssonShahgholian}, it holds
       \[  |\nabla u(x_0)|\leq C\bigg[ \frac{1}{r}\mean{\partial B_r(x_0)}u(x)dx+ r^3\bigg]\leq C\]
       for a universal constant $C$.        
\end{proof}
	\begin{lemma}\label{le:deinsity estimate}
         There exists a universal constant $q_4\in (0, q_3]$ such that for all  $q\in (0,q_4 ]$, calling $\Om$ an optimal set for problem~\eqref{eq:min mathcal E q,M,eta} and $u$ a positive normalized function attaining $E_{q,M,\eta}(\Omega)$, there exist {positive} constants $\theta= \theta(\eta)$ and $\rho_0=\rho_0(\eta)<1$ such that for every $x_0\in \partial \Omega$ and every $\rho\leq \rho_0$, we  have
		\begin{equation}\label{eq:4.30}
			\theta\leq \frac{|\Omega\cap B_\rho(x_0)|}{|B_\rho|}{\leq (1-\theta)}.
		\end{equation}
	\end{lemma}
\begin{proof}
    The proof follows from \cite[Lemma 3.12]{MazzoleniMuratovRuffini}.
\end{proof}

\begin{lemma}\label{le: convergence hausdorff}
           For all $\delta>0$ there exists $q_\delta=q_\delta(\eta)\in(0,q_4]$ such that for all $q\leq q_\delta$, we have \[ {\rm dist}_H(\partial \Om_{{q,M,\eta}},\partial
            B_{{q,M,\eta}})\leq \delta.
		\] 
        where $\Om_{q,M,\eta}$ is an optimal domain for problem \eqref{eq: R q,M,eta in B_R}.
\end{lemma}
\begin{proof}
    It is a trivial adaptation of \cite[Lemmas 3.14,3.15,3.16]{MazzoleniMuratovRuffini} hence we just sketch the proof. By the quantitative version of the Faber-Krahn inequality we first deduce that minimizers are close in $L^1-$topology to a ball for $q-$small. Such an $L^1-$closeness can then be improved into a proximity in Hausdorff distance by means of the uniform density estimates just obtained.
\end{proof}

\section{Optimality condition and improvements of flatness}\label{sec:improvementofflatness}
We show now the following result.
	\begin{theorem}\label{thm:agalcathm2}
          Let $q \in (0,q_4]$, {let} $\Omega$ be a minimizer of
          \eqref{eq: inf EqOmega}, and {let} $u$ be an optimal nonnegative function
          attaining $E_q(\Omega)$. Then we have that:
		\begin{enumerate}
			\item[(i)] There is a Borel function $\mu_{u}\colon \partial \Om \rightarrow \R$ such that, in the sense of the distributions, one has
			\begin{equation}\label{eq:bdp4.31}
				-\Delta u = \lambda_qu - q v_u - \mu_{u}\mathcal H^{2}\resmeas\partial\Omega,\qquad\text{ in }B_R.
			\end{equation}
			\item[(ii)]  There exist constants $0<c<C<+\infty$, depending on $R$,  such that $c\le \mu_{u}\le C$.
			\item[(iii)] For all points $\overline x\in \partial^*\Om =\partial^*\{u>0\}$, the measure theoretic inner unit normal $\nu_{u}(\overline x)$ is well defined and, as $\rho\to0$, 
			\begin{equation}\label{eq:bdp4.32}
				\frac{\Om-\overline x}{\rho}\rightarrow \{x : x\cdot \nu_{u}(\overline x)\geq 0\},\qquad \text{in }L^1(B_R).
			\end{equation}
			\item[(iv)]For $\mathcal H^{2}$ almost all $\overline x\in \partial^*\{u>0\}$ we have 
			\begin{equation}\label{eq:bdp4.33}
				\frac{u(\overline x+\rho x)}{\rho}\longrightarrow \mu_{u}(\overline x)(x\cdot \nu_{u}(\overline x))_+,\qquad \text{in }W^{1,p}(B_R)\;\text{for every }p\in[1,+\infty).
			\end{equation}
			\item[(v)] $\mathcal H^{2}(\partial \Om\setminus\partial^*\Om)=0$.
		\end{enumerate}
        Moreover $\mu_u:\partial \Omega\to \R$ is constant on $\partial^*\Omega$.
	\end{theorem}
    \begin{proof}
        	
		The proof is essentially identical to that in \cite[Section	4]{AltCaffarelli}. We only have to check that our hypotheses match with those in \cite{AltCaffarelli}.  First by Theorem \ref{thm: pde and L infty estimate}, $u$ satisfies
		\[
		-\Delta u-Q(x)=0 {\qquad \text{in} \ \mathcal D'(\Omega),} 
		\]
		where
		$Q=\lambda_qu- q v_u\in L^\infty(\Omega)$ and
		$u\in H^1_0(\Omega)$. Hence, by repeating the proof of \cite[Theorem 4.5]{AltCaffarelli} or by directly applying
		\cite[Proposition 2.3]{BucurMazzoleniPratelliVelichkov} one obtains that there exists a  positive Radon measure concentrated on $\partial\Omega$ that we denote $\mu_u\mathcal H^2\resmeas\partial\Omega$. Moreover, {thanks to the non-degeneracy, see \cite[Remark 2.8]{MazzoleniTTerraciniVelichkov2017}, and the Lipschtiz continuity of $u$} we have that there exist constant $C>c>0$ depending on $q$ and $R$ such that 
		\[
		c\le{\frac1r \mean{\partial B_r}{u\,d\mathcal H^2}}\le C.
		\]
		Hence we can work under the hypotheses of \cite[Theorem 4.5]{AltCaffarelli} so that $\mu_u$ is a density of a Radon measure on $\partial\Omega$ and, denoting still with $\mu_u$ the function defininig it, $\mu_u$ satisfies $(i)-(v)$.

            Let us now prove that $\mu_u$ is constant on $\partial^*\Omega$.  The proof follows the path of \cite[Theorem 6.5]{MazzoleniRuffini}, in turn inspired by \cite{AguileraAltCaffarelli}. 
		Due to the nonlocal term, we will have to perform some new and non-straightforward computations.
		
		We reason by contradiction and we assume that  there exists $x_0,x_1\in\partial^*\Omega$ such that
		\[
		\mu_u(x_0) < \mu_u(x_1).
		\] 
		Then we construct a family of volume preserving diffeomorphisms  as follows: let $\kappa<1$ and $\rho<1$  and let $\varphi\in C^1_0(B_1(0))$ be a non-null, radially symmetric function supported in $B_1(0)$. We define 
		\[
		\tau_{\rho,\kappa}(x)=\tau(x)=
		x+\sum_{i\in\{0,1\}}(-1)^i\kappa\rho\varphi\left(\frac{|x-x_i|}{\rho}\right)\nu_{x_i} \chi_{B_\rho(x_i)},
		\]
		where $\nu_{x_i}$ are the measure theoretic inner normals to $\partial^*\Omega$ at $x_i$, $i=1,2$.
		
		It is easy to notice that $\tau$ is indeed a diffeomorphism for $\rho$ and $\kappa$ small enough and that   $\tau(x)-x$ vanishes outside $B_\rho(x_0)\cup B_\rho(x_1)$. Moreover we have:
		\begin{equation}\label{eq:dettau}
			\nabla\tau(x)=Id+\sum_{i\in\{0,1\}}(-1)^i\kappa\varphi'\left(\frac{|x-x_i|}{\rho}\right)\frac{x-x_i}{|x-x_i|}\otimes\nu_{x_i} \chi_{B_\rho(x_i)}, 
		\end{equation}
		so that\footnote{We are using the formula $\det(Id+\xi A)=1+   trace(A)\xi+o(\xi)$ for a matrix $A\in\R^{3\times 3}$.}
		\begin{equation}\label{eq:espdet}
			\det(\nabla\tau(x)) = 1+ \sum_{i\in\{0,1\}}(-1)^i\kappa\varphi'\left(\frac{|x-x_i|}{\rho}\right)\frac{x-x_i}{|x-x_i|}\cdot\nu_{x_i} \chi_{B_\rho(x_i)}+o(\kappa).
		\end{equation}
		We call $\Omega_\rho =\tau(\Omega)$. We aim to show that for  $\kappa,\rho$ small enough we obtain a contradiction with the minimality of $\Omega$. To do that, we deal with the first variation of each term of the sum defining $E_{q,M,\eta}$, see Remark~\ref{rmk:etafissato} for the choice of $\eta$. We stress that the computations regarding the volume and the Dirichlet energy contributions are identical to those performed originally in \cite{AguileraAltCaffarelli} (see also \cite{BrascoDePhilippisVelichkov} and \cite{DePhilippisMariniMukoseeva}, where the same idea is applied). Moreover, exactly as in the proof of \cite[Theorem 6.5]{MazzoleniRuffini} one obtains that 
		\begin{equation}\label{variazionevolume}
			f_\eta(\Omega_\rho)-f_\eta(\Omega)= o(\rho^3),\qquad \text{as }\rho\rightarrow 0,
		\end{equation}
		and that
		\begin{equation}\label{eq:variazionetorsione}
			\frac{1}{\rho^3} \Big(\int_{\Omega_\rho}|\nabla \widetilde u_\rho|^2\, dx-\int_\Omega |\nabla u|^2\, dx\Big)\le\kappa{ (\mu_u^2(x_0)-\mu_u^2(x_1))}C(\varphi) + o_\rho(1)+o(\kappa),
		\end{equation}
		where $u$ is the function attaining $E_q(\Omega)=E_{q,M,\eta}(\Omega)$, $\widetilde u_\rho=u \circ\tau^{-1}$ and 
		\begin{equation}\label{rfk}
			C(\varphi)=\int_{B_1(0)\cap\{y\cdot\nu=0\}}\varphi(|y|)\,d\mathcal H^{2}(y)=-\int_{B_1(0)\cap \{y\cdot \nu>0\}}\varphi'(|y|)\frac{y\cdot \nu}{|y|}\,dy,
		\end{equation}
		with the last equality that follows from the Divergence Theorem, recalling that $\nu$ is a inner normal and $${\mathrm{div}}(\varphi(|y|)\nu)=\varphi'(|y|)\frac{y\cdot \nu}{|y|}.$$
		Notice also that by the radial symmetry of $\varphi$  the value of $C(\varphi)$ is not affected by the choice of $\nu$.
		Moreover 
    \begin{align}\label{eq:variazione norma L^2}
 \frac{1}{\rho^3} \Bigg|\int_{\Omega_\rho} \widetilde u^2_\rho -\int_\Omega u^2 \Bigg|=& \frac{1}{\rho^3} \Bigg|\int_\Om (u^2 \text{ det}\nabla\tau -u^2)dx \Bigg|    \\
 =& \Bigg| \sum_{i\in \{0,1\}} \int_{B_1(0)\cap \left(\frac{\Omega-x_i }{\rho}\right)} u^2(x_i+\rho y) \text{ det}\nabla\tau(x_i+\rho y) -u^2(x_i+\rho y)dy\Bigg| \\
 =& \Bigg| \sum_{i\in \{0,1\}} \int_{B_1(0)\cap \left(\frac{\Omega-x_i }{\rho}\right)} (-1)^{i}\kappa\varphi'(|y|)\frac{y}{|y|} \frac{u^2(x_i+\rho y)}{\rho^2}\rho^2dy\Bigg|+o(\kappa)\\ =& o(\rho)+o(\kappa)
    \end{align}
    where we performed the change of variable $x=x_i+\rho y$, we used \eqref{eq:espdet} and Theorem \ref{thm:agalcathm2} points (\textit{iii}) and (\textit{iv}).    
    
    We are left to compute the variation of the nonlocal term $D(\cdot,\cdot)$. We claim that 
    \begin{equation}\label{eq:variazione D(u,u)}
        \frac{1}{\rho^3} (D(u_\rho,u_\rho)-D(u,u)=o(\kappa)+o_\rho(1).        
    \end{equation}
    Once \eqref{eq:variazione D(u,u)} is proved, the conclusion follows: by minimality of $\Om$, \eqref{eq:variazionetorsione}, \eqref{eq:variazione norma L^2} and \eqref{eq:variazione D(u,u)} (recalling also that $E_{q,M,\eta}(\Omega_\rho)\leq E_{q,M,\eta}(\Omega_\rho,\widetilde u _\rho)$) it holds
    \begin{align}
        0\leq & E_{q,M,\eta}(\Omega_\rho)-E_{q,M,\eta}(\Omega)\\
        \leq & \kappa\rho^3C(\varphi) \Bigg( (\mu_u(x_0))^2-(\mu_u(x_1))^2 \Bigg)+o(\rho^3)+\rho^3o(\kappa).
    \end{align}
    Since we suppose $(\mu_u(x_0))^2-(\mu_u(x_1))^2 <0$, then we obtain a contradiction as soon as $\rho$ and $\kappa$ are small enough.

    It remains to show \eqref{eq:variazione D(u,u)}. Setting $w(x):=v_{u}(x)u(x)$ and $\widetilde w(x)=v_{\widetilde u_\rho}(x) \widetilde u_\rho (x)$ and recalling \eqref{eq:espdet}, it holds
    \begin{align}\label{eq: variazione D(u,u), 1}
        \frac{1}{\rho^3}\Bigg( &D(u_\rho,u_\rho)- D(u,u)\Bigg)=\frac{1}{\rho^3}\Bigg( \int_{\Om_\rho}\widetilde w  - \int_\Om w \Bigg)= \frac{1}{\rho^3}\Bigg( \int_{\Om}\widetilde w (\tau(x))\text{ det}\nabla\tau(x)  -  w(x)dx  \Bigg)
        \\ \leq &\frac{1}{\rho^3} \int_{\Om}\widetilde (w\circ \tau - w)dx+ \frac{1}{\rho^3}\int_\Om \Bigg( \sum_{i\in \{ 0,1\}} (-1)^i\kappa \varphi'\Big(\frac{|x-x_i|}{\rho}\Big)\frac{x-x_i}{|x-x_i|}\cdot \nu_{x_i} \chi_{B_\rho(x_i)}+o(\kappa) \Bigg)\widetilde w\circ \tau dx.
    \end{align}
    Recalling that $v_{\widetilde u_\rho}$ is uniformly bounded (Theorem~\ref{thm: pde and L infty estimate}) and $\widetilde u_\rho$ is Lipschitz continuous (Lemma~\ref{le:lipschitz}), then $|\widetilde w(\tau(x))|\leq C \rho$ in $\Omega\cap B_{\rho}(x_i)$ since $\widetilde u_\rho(\tau(x_i))=u(x_i)=0$ as $x_i\in \partial^*\Om$, for some universal constant $C>0$. This implies that
    \begin{align}\label{eq: variazione D(u,u), 2}
\Bigg|  \frac{1}{\rho^3}&\int_\Om \Bigg( \sum_{i\in \{ 0,1\}} (-1)^i\kappa \varphi'\Big(\frac{|x-x_i|}{\rho}\Big)\frac{x-x_i}{|x-x_i|}\cdot \nu_{x_i} \chi_{B_\rho(x_i)} \Bigg)\widetilde w\circ \tau dx\Bigg|\leq \frac{C}{\rho^2}|B_\rho|=o_\rho(1).
    \end{align}
    Moreover 
    \begin{align}
        \frac{1}{\rho^3} \int_{\Om}\widetilde w\circ \tau - w\, dx=&\frac{1}{\rho^3} \int_\Om u(x)\Bigg( \int_{\Om_\rho} \frac{\widetilde u_\rho(y)}{|\tau (x)-y|}dy-\int_\Om \frac{u(y)}{|x-y|}dy \Bigg)dx\\=&\frac{1}{\rho^3} \int_\Om u(x)\Bigg( \int_{\Om} \frac{ u(y)}{|\tau (x)-\tau(y)|}\text{ det}\nabla\tau dy-\int_\Om \frac{u(y)}{|x-y|}dy \Bigg)dx
        \\ =& \frac{1}{\rho^3}\int_\Om u(x)\int_{\Omega\cap (B_\rho(x_0)\cup B_\rho(x_1))}u(y)\Bigg( \frac{1}{|\tau(x)-\tau(y)|}-\frac{1}{|x-y|}\Bigg)dydx\\
        &+\frac{1}{\rho^3}\int_\Om u(x)\int_{\Omega} \frac{u(y)}{|\tau(x)-\tau(y)|}\Bigg( \sum_{i\in \{ 0,1\}} (-1)^i\kappa \varphi'\Big(\frac{|y-x_i|}{\rho}\Big)\frac{y-x_i}{|x-x_i|}\cdot \nu_{x_i} \chi_{B_\rho(x_i)}+o(\kappa)  \Bigg)dydx
    \end{align}
    By computations analogous to the ones in \eqref{eq: variazione D(u,u), 2}, it holds 
    \begin{align}
        \Bigg| \frac{1}{\rho^3}\int_\Om u(x)\int_{\Omega}& \frac{u(y)}{|\tau(x)-\tau(y)|}\Bigg( \sum_{i\in \{ 0,1\}} (-1)^i\kappa \varphi'\Big(\frac{|y-x_i|}{\rho}\Big)\frac{y-x_i}{|x-x_i|}\cdot \nu_{x_i} \chi_{B_\rho(x_i)}  \Bigg)dydx \Bigg|\\ \leq& \frac{1}{\rho^3} \int_\Om u(x)\int_{\Omega\cap (B_\rho(x_0)\cup B_\rho(x_1))} \frac{u(y)}{|\tau(x)-\tau(y)|}\leq \frac{C}{\rho^2}\int_\Om \int_{B_\rho(x_0)\cup B_\rho(x_1)} \frac{1}{|\tau(x)-\tau(y)|}dydx\\\leq& \frac{C}{\rho^2}\int_{\Omega_\rho}\int_{B_\rho(x_0)\cup B_\rho(x_1)}\frac{1}{|x-y|}\leq \frac{C}{\rho^2}|B_\rho|=o_\rho(1)
    \end{align}
    where we used Theorem \ref{thm: pde and L infty estimate} and \cite[Lemma 2.4]{FuscoPratelli}. Eventually we obtain

    \begin{align}\label{eq: variazione D(u,u), 3}
        \Bigg| \frac{1}{\rho^3} \int_{\Om}\widetilde w\circ \tau - wdx\Bigg|\leq& \frac{1}{\rho^3}\Bigg|\int_\Om u(x)\int_{\Omega\cap (B_\rho(x_0)\cup B_\rho(x_1))}u(y)\Bigg( \frac{1}{|\tau(x)-\tau(y)|}-\frac{1}{|x-y|}\Bigg)dydx \Bigg|+o_\rho(1)
        \\ \leq& \frac{C}{\rho^2}\int_\Om \int_{ B_\rho(x_0)\cup B_\rho(x_1)}\Bigg( \frac{1}{|\tau(x)-\tau(y)|}+\frac{1}{|x-y|}\Bigg)dydx +o_\rho(1)=o_\rho(1)
    \end{align}
    where the last equality holds again by \cite[Lemma 2.4]{FuscoPratelli}. By \eqref{eq: variazione D(u,u), 1}, \eqref{eq: variazione D(u,u), 2} and \eqref{eq: variazione D(u,u), 3} we deduce \eqref{eq:variazione D(u,u)} and this concludes the proof.
    \end{proof}

	We are now in position to show $C^{2,\alpha}-$regularity of   the boundary of a minimizer $\Om$. This can be done in two steps: first one shows that such a boundary is locally the graph of a $C^{2,\alpha}$ function defined on the boundary of a ball. To do that one exploits the improvement of flatness technique from~\cite[Section~7 and~8]{AltCaffarelli}, readapted with minimal changes to our setting with a right hand side as in~\cite[Appendix]{GustafssonShahgholian}. Then, as we already know by the previous section that the boundary of $\Om$ is close in Hausdorff distance to that of a ball, we obtain that the local parametrization is a global parametrization of class $C^{2,\alpha}$ on the boundary of the ball.
	{We first need a definition (see~\cite[Definition~7.1]{AltCaffarelli}).
		\begin{definition}
			Let $\gamma_\pm\in(0,1]$ and $k>0$. A weak solution $u$ of~\eqref{eq:bdp4.31} is of class $F(\gamma_-,\gamma_+,k)$ in $B_\rho(x_0)$ with respect to direction $\nu\in \mathbb{S}^{N-1}$ if 
			\begin{enumerate}
				\item[(a)]$x_0\in \partial \{u>0\}$ and \[
				\begin{split}
					u=0,\qquad &\text{for }(x-x_0)\cdot \nu\leq -\gamma_-\rho,\quad x\in B_\rho(x_0),\\
					u(x)\geq \mu_{u}(x_0)[(x-x_0)\cdot \nu-\gamma_+\rho],\qquad &\text{for }(x-x_0)\cdot \nu\geq \gamma_+\rho,\quad x\in B_\rho(x_0).
				\end{split}
				\]
				\item[(b)] $|\nabla u(x_0)|\leq \mu_u(x_0)(1+k)$ in $B_\rho(x_0)$ and ${\rm osc}_{B_{\rho}(x_0)}\mu_u\leq k\mu_u(x_0)$.
			\end{enumerate} 
		\end{definition}
		We note that when $k=+\infty$, then condition $(b)$ is automatically satisfied.}
	We can show the following result.
	\begin{theorem}\label{thm:bdvthm4.18}
          Let $q\in (0,q_4]$, $\Omega$ be an optimal set
          for~\eqref{eq: inf EqOmega}, and $u$ a function attaining
          $E_q(\Omega)$ and a weak solution to~\eqref{eq:bdp4.31} in
          $B_R$.
		Then there
		are constants $\overline \gamma$ and $\overline k$, depending only on
		$R$, $ \mu_u$, such that
		if $u$ is of class $F(\gamma,1,+\infty)$ in $B_{4\rho}(x_0)$ with respect to some direction $\nu\in \mathbb{S}^{N-1}$ with $\gamma\leq \overline \gamma$ and $\rho\leq \overline k \gamma^2$,
		then there exists a $C^{2, \alpha}$ function
		$f\colon \R^{2}\to \R$ with
		$\|f\|_{C^{2,\alpha}}\leq C(R, \mu_u)$ such that,
		calling
		\[ {\rm graph}_\nu f:=\{x\in \R^3 : x\cdot \nu
		=f(x-(x\cdot\nu)\nu)\},
		\]
		then \[
		\partial \{u>0\}\cap B_\rho(x_0)=(x_0+{\rm graph}_\nu(f))\cap B_\rho(x_0).
		\]
		Moreover for all $\eps_0>0$ there exists $q_\eps\in
                (0,q_4 ]$ such that if $q<q_\eps$ then
		\[
		\partial \{u>0\}=\left\{\left(r+\varphi\left(\tfrac{x}{|x|}\right)\right)\frac{x}{|x|}\,:\, x\in\partial B_r  \right\}
		\]
		where $\varphi\colon \partial B_1\to \R$ is a function with the same regularity of $f$ and $\|\varphi\|_{C^{2,\alpha}}\leq \eps_0$.
	\end{theorem}
	We omit the proof which is identical to that in \cite[Theorems 1.2 and
	6.8]{MazzoleniRuffini} (which is in turn inspired by~\cite[Theorem~8.1]{AltCaffarelli} and~\cite[Theorem~2.17 and Appendix]{GustafssonShahgholian}). {We note that in our setting $\mu_u$ is constant, thus the requirement to be $C^{1,\alpha}$ regular is trivially satisfied.} 
	
	We are now in position to prove Theorem~\ref{thm:main}.
	\begin{proof}[Proof of Theorem~\ref{thm:main}]
		The existence of a minimizer follows from Theorem~\ref{thm:existenceEq,M,eta} and Theorem~\ref{thm:noconstraint}.
		On the other hand, the fact that any optimal set is $C^{2,\alpha}$ nearly spherical follows from Theorem~\ref{thm:bdvthm4.18}.
	\end{proof}

\appendix
\section{Remarks on $D(\cdot,\cdot)$}\label{appendix}
In this Appendix we show how the minimization of the nonlocal term $D(\cdot,\cdot)$ alone is ill posed. 
\begin{lemma}\label{le: inf D(u,u) non attained}
We have that
	\[\inf\{D(u,u):\; u\in L^2(\R^3), \;||u||_{L^2(\R^3)}=1\}=0 \]
and it is not attained.
\begin{proof}
Let us show the existence of a sequence $(\phi_\eps)\subset L^2(\R^3)$ such that $||\phi_\eps||_{L^2(\R^3)}=1$ and $D(\phi_\eps,\phi_\eps)\to 0$ as $\eps \to 0$.  Let $\varphi\in C_c^\infty(\R^3)$, $\varphi\geq 0$, $\Omega:=\supp (\phi) \subset B(0,1)$ and $||\phi||_{L^2(\Om)}=1$. Let $\eps>0$, then we define 
\[ \phi_\eps(x):=\eps^{-3/2}\phi\Big(\frac{x-x_0}{\eps}\Big),\]
so $||\phi_\eps||_{L^2(\R^3)}=1$ and $\phi_\eps\in L^2(\R^3)$. Evaluating the energy we obtain 
\[ D(\phi_\eps,\phi_\eps)=\eps^2\int_\Om\int_\Om \frac{\phi(x)\phi(y)}{|x-y|}\; dxdy.\]
Thus as $\eps\to 0$,  $D(\phi_\eps,\phi_\eps) \to 0$, so we proved that the infimum under study is equal to zero.

To show that the infimum is not attained, it is enough to notice that any admissible $\tilde{u}\in L^2(\R^3)$ must be $\tilde u\ne0$ by the $L^2$ constraint. Hence $D(\tilde u,\tilde u)>0$ by Lemma~\ref{le:positivityD(u,u)}. 
\end{proof}
\end{lemma}
\begin{remark}
 Let $\Omega$ be a (quasi-)open set in $\R^3$ such that there exists a point $x_0$ and a radius $r>0$ for which $B_r(x)\subseteq \Omega$, then by the same proof of Lemma \ref{le: inf D(u,u) non attained} we deduce that 
	\[\inf\{D(u,u):\; u\in \H(\Om), ||u||_{L^2(\Omega)}=1\} = \inf\{D(u,u):\; u\in \H(\Om), ||u||_{L^2(\Omega)}=1,\; u\geq 0\}=0.\]
\end{remark}
Lemma~\ref{le: inf D(u,u) non attained} does not really show the homogenization phenomenon that we expect, essentially because of scaling properties of the functional. We can impose an additional $L^\infty$ bound to see the importance of considering sign-changing functions.
\begin{lemma} For all $L\geq 1$, then
\[ \inf_{\substack{\Omega \subseteq \R^3, \, \text{quasi-open}\\ |\Omega|=|B_1|}} \inf\left\{ D(u,u): u\in \H(\Omega), \int_{\Omega}u^2dx=1, ||u||_{L^\infty(\Omega)}\leq L\right\}=0.\]
\end{lemma}
\begin{proof}
	Take $\phi\in C_c^\infty(B_1)$ such that $||\phi||_{L^2(B_1)}=1$ and $||\phi||_{L^\infty(B_1)}\leq 1$. Consider $n\in \N$, $\eps=(2n)^{-1/3}$ and define
	\[ \phi_{\eps,n}(x):=\frac{1}{\sqrt{2n}}\left( \sum_{i=1}^{n}\eps^{-3/2}\phi \Big(\frac{x-x_i}{\eps}\Big) - \sum_{i=n+1}^{2n}\eps^{-3/2}\phi \Big(\frac{x-x_i}{\eps}\Big)\right),  \]
	where $x_i\in\R^3$ such that $\supp\left\{\phi \Big(\frac{\cdot-x_i}{\eps}\Big)\right\}\cap  \supp\left\{\phi \Big(\frac{\cdot-x_j}{\eps}\Big)\right\}=\emptyset$ whenever $i\ne j$ and $\min\{|x_i-x_j| : i\not=j\}$ is diverging  as $n$ diverges. Then it is easy to see that $||\phi_{\eps,n}||_{L^2(\R^3)}=1$ and
	\[ |\supp\{ \phi_{\eps,n}\}|= 2n|B_\eps|=2n\eps^3|B_1|=|B_1|.\]
	 Moreover  
	\[ ||\phi_{\eps,n}||_{L^\infty(\R^3)} \leq \frac{\eps^{-3/2}}{\sqrt{2n}}= 1,\]
	so $\phi_{\eps,n}$ satisfies the $L^\infty$ constraint since true since $L\geq 1$. Then arguing as in Lemma \ref{le: inf D(u,u) non attained}, 
	\[ D(\phi_{\eps,n},\phi_{\eps,n})=n^{1/2}(2n)^{-2/3}D(\phi,\phi)\to 0, \]
	as $n$ diverges. Since 	\[ \inf_{\substack{\Omega \subseteq \R^3, \, \text{quasi-open}\\ |\Omega|=1}} \inf\left\{ D(u,u): u\in \H(\Omega), \int_{\Omega}u^2dx=1, ||u||_{L^\infty(\Omega)}\leq L\right\} \leq D(\phi_{\eps,n},\phi_{\eps,n}) \to 0,\]
	then by Lemma \ref{le:positivityD(u,u)} we conclude.         
\end{proof}

\paragraph{\bf Acknowledgements} The authors have been partially supported by the MUR and the European Union via PRIN2022 project P2022R537CS.
The authors are members of
INdAM-GNAMPA and have been partially supported by the INdAM-GNAMPA projects CUP E5324001950001 and CUP E53C25002010001.  \\

\bibliographystyle{abbrv}
\bibliography{bibliografia}

@article {Weinberger,
    AUTHOR = {Weinberger, H. F.},
     TITLE = {Remark on the preceding paper of {S}errin},
   JOURNAL = {Arch. Rational Mech. Anal.},
    VOLUME = {43},
      YEAR = {1971},
     PAGES = {319--320},
}

@book {DalMaso,
    AUTHOR = {Dal Maso, Gianni},
     TITLE = {An introduction to {$\Gamma$}-convergence},
    SERIES = {Progress in Nonlinear Differential Equations and their Applications},
    VOLUME = {8},
 PUBLISHER = {Birkh\"auser Boston Inc. Boston, MA},
      YEAR = {1993},
}

@article {Bucur,
    AUTHOR = {Bucur, Dorin},
     TITLE = {Minimization of the {$k$}-th eigenvalue of the {D}irichlet
              {L}aplacian},
   JOURNAL = {Arch. Ration. Mech. Anal.},
  FJOURNAL = {Archive for Rational Mechanics and Analysis},
    VOLUME = {206},
      YEAR = {2012},
    NUMBER = {3},
     PAGES = {1073--1083},
}

@article{FigalliZhang,
      title={Serrin's overdetermined problem in rough domains}, 
      author={Alessio Figalli and Yi Ru-Ya Zhang},
       JOURNAL = {J. Eur. Math. Soc., in press },
      year={2025},
}

@article {BenguriaPereira,
    AUTHOR = {Benguria, Rafael D. and Pereira, Marcone C.},
     TITLE = {Remarks on the spectrum of a non-local {D}irichlet problem},
   JOURNAL = {Bull. Lond. Math. Soc.},
  FJOURNAL = {Bull. Lond. Math. Soc.},
    VOLUME = {53},
      YEAR = {2021},
    NUMBER = {6},
     PAGES = {1898--1915},
}

@book{GilbargTrudinger,
	author = {Gilbarg, D. and Trudinger, N.S.},
	title = {Elliptic Partial Differential Equations of Seconds Order},
	journal = {Classics in Mathematics},
	volume = {224},
	publisher = {Springer Berlin, Heidelberg},
	year = {2001}
}

@book{Mazya,
	title={Sobolev Spaces: with Applications to Elliptic Partial Differential Equations},
	author={Maz'ya, V.},
	journal={Grundlehren der mathematischen Wissenschaften},
	year={2011},
	volume={342},
	publisher={Springer Berlin Heidelberg}
}

@book{LiebLoss,
	title={Analysis},
	author={Lieb, E.H. and Loss, M.},
	journal={Graduate Studies in Mathematics},
	volumne={14},
	year={2001},
	publisher={American Mathematical Society}
}

@article{AguileraAltCaffarelli,
	author = {Aguilera, N. and Alt, H. W. and Caffarelli, L. A.},
	title = {An optimization problem with volume constraint},
	journal = {SIAM J. Control Optim.},
	volume = {24},
	year = {1986},
	number = {2},
	pages = {191-198},
}

@article {FuscoPratelli,
    AUTHOR = {Fusco, Nicola and Pratelli, Aldo},
     TITLE = {Sharp stability for the {R}iesz potential},
   JOURNAL = {ESAIM Control Optim. Calc. Var.},
    VOLUME = {26},
      YEAR = {2020},
     PAGES = {Paper No. 113, 24},
}

@article {MazzoleniPratelli,
    AUTHOR = {Mazzoleni, Dario and Pratelli, Aldo},
     TITLE = {Existence of minimizers for spectral problems},
   JOURNAL = {J. Math. Pures Appl. (9)},
  FJOURNAL = {J. Math. Pures Appl.},
    VOLUME = {100},
      YEAR = {2013},
    NUMBER = {3},
     PAGES = {433--453},
}

@article {MazzoleniRuffini,
    AUTHOR = {Mazzoleni, Dario and Ruffini, Berardo},
     TITLE = {A spectral shape optimization problem with a nonlocal
              competing term},
   JOURNAL = {Calc. Var. Partial Differential Equations},
    VOLUME = {60},
      YEAR = {2021},
    NUMBER = {3},
     PAGES = {114},
}

@article{BucurMazzoleni,
  title = {A Surgery Result for the Spectrum of the Dirichlet Laplacian},
  volume = {47},
  number = {6},
  journal = {SIAM J. Math. Anal.},
  author = {Bucur,  Dorin and Mazzoleni,  Dario},
  year = {2015},
  pages = {4451-4466}
}

@article{BucurButtazzo,
  title = {On the characterization of the compact embedding of Sobolev spaces},
  volume = {44},
  journal = {Calc. Var. Partial Differential Equations},
  author = {Bucur,  Dorin and Buttazzo,  Giuseppe},
  year = {2011},
  pages = {455–475}
}

@book{VelichkovLecture,
  title = {Regularity of the One-phase Free Boundaries},
  journal = {Lecture Notes of the Unione Matematica Italiana},
  publisher = {Springer International Publishing},
  author = {Velichkov,  Bozhidar},
  year = {2023}
}

@article{BrianconHayouniPierre,
  title = {Lipschitz continuity of state functions in some optimal shaping},
  volume = {23},
  number = {1},
  journal = {Calc. Var. Partial Differential Equations},
  author = {Brian\c{c}on,  Tanguy and Hayouni,  Mohammed and Pierre,  Michel},
  year = {2005},
  pages = {13–32}
}

@article{GustafssonShahgholian,
  volume = {1996},
  author={Gustafsson, Bj\"orn and Shahgholian, Henrik},
  number = {473},
  journal = {J. Reine Angew. Math.},
  year = {1996},
  pages = {137–180},
  title={Existence and geometric properties of solutions of a free boundary problem in potential theory}
}

@article{AltCaffarelli,
author = {Alt, H.~W. and Caffarelli, L.~A.},
journal = {J. Reine Angew. Math.},
pages = {105-144},
title = {Existence and regularity for a minimum problem with free boundary.},
volume = {325},
year = {1981}
}

@article{BucurMazzoleniPratelliVelichkov,
  title = {Lipschitz Regularity of the Eigenfunctions on Optimal Domains},
  volume = {216},
  number = {1},
  journal = {Arch. Ration. Mech. Anal.},
  author = {Bucur,  Dorin and Mazzoleni,  Dario and Pratelli,  Aldo and Velichkov,  Bozhidar},
  year = {2014},
  pages = {117–151}
}

@article{DePhilippisMariniMukoseeva,
  title = {The sharp quantitative isocapacitary inequality},
  volume = {37},
  number = {6},
  journal = {Rev. Mat. Iberoam.},
  author = {De Philippis,  Guido and Marini,  Michele and Mukoseeva,  Ekaterina},
  year = {2021},
  pages = {2191–2228}
}

@book{Giaquinta,
	title={Multiple Integrals in the Calculus of Variations and non linear Elliptic Systems},
	author={M. Giaquinta},
	year={1983},
	publisher={Princeton University Press}
}

@article{MazzoleniTTerraciniVelichkov2017,
  title = {Regularity of the optimal sets for some spectral functionals},
  volume = {27},
  number = {2},
  journal = {Geom. Funct. Anal.},
  author = {Mazzoleni,  Dario and Terracini,  Susanna and Velichkov,  Bozhidar},
  year = {2017},
  pages = {373–426}
}

@article{FiMaPra2010,
  title={A mass transportation approach to quantitative isoperimetric inequalities},
  author={Figalli, Alessio and Maggi, Francesco and Pratelli, Aldo},
  journal={Invent. Math.},
  volume={182},
  number={1},
  pages={167--211},
  year={2010},
  publisher={Springer}
}

@article{FuMaPra2008,
  title={The sharp quantitative isoperimetric inequality},
  author={Fusco, Nicola and Maggi, Francesco and Pratelli, Aldo},
  journal= {Ann. of Math. (2)},
  pages={941--980},
  year={2008},
  publisher={JSTOR}
}

@article{MazzoleniMuratovRuffini,
  title={An optimal design problem for a charge qubit},
  author={Mazzoleni, Dario and Muratov, Cyrill B and Ruffini, Berardo},
  journal={Comm. Partial Differential Equations},
  volume={50},
  pages={1029--1073},
  year={2025},
  publisher={Taylor \& Francis}
}

@article{ChoksiPeletier,
  title={Small volume fraction limit of the diblock copolymer problem: {I}. {S}harp-interface functional},
  author={Choksi, Rustum and Peletier, Mark A},
  journal={SIAM J. Math. Anal.},
  volume={42},
  number={3},
  pages={1334--1370},
  year={2010},
  publisher={SIAM}
}

@article{KM1,
  title={On an isoperimetric problem with a competing nonlocal term {I}: {T}he planar case},
  author={Kn{\"u}pfer, Hans and Muratov, Cyrill B},
  journal={Comm. Pure Appl. Math.},
  volume={66},
  number={7},
  pages={1129--1162},
  year={2013},
  publisher={Wiley Online Library}
}

@article{KM2,
  title={On an isoperimetric problem with a competing nonlocal term {II}: {T}he general case},
  author={Kn{\"u}pfer, Hans and Muratov, Cyrill B},
  journal={Comm. Pure Appl. Math.},
  volume={67},
  number={12},
  pages={1974--1994},
  year={2014},
  publisher={Wiley Online Library}
}

@article{CiCSpa,
  title={Droplet minimizers of an isoperimetric problem with long-range interactions},
  author={Cicalese, Marco and Spadaro, Emanuele},
  journal={Comm. Pure Appl. Math.},
  volume={66},
  number={8},
  pages={1298--1333},
  year={2013},
  publisher={Wiley Online Library}
}

@article{CicLeo2012,
  title={A selection principle for the sharp quantitative isoperimetric inequality},
  author={Cicalese, Marco and Leonardi, Gian Paolo},
  journal={Arch. Ration. Mech. Anal.},
  volume={206},
  number={2},
  pages={617--643},
  year={2012},
  publisher={Springer}
}

@article{GNR1,
  title={Existence and stability for a non-local isoperimetric model of charged liquid drops},
  author={Goldman, Michael and Novaga, Matteo and Ruffini, Berardo},
  journal={Arch. Ration. Mech. Anal.},
  volume={217},
  number={1},
  pages={1--36},
  year={2015},
  publisher={Springer}
}

@article{GNR2025,
  title={Rigidity of the ball for an isoperimetric problem with strong capacitary repulsion},
  author={Goldman, Michael and Novaga, Matteo and Ruffini, Berardo},
  journal={J. Eur. Math. Soc.},
  volume={27},
  number={10},
  pages={4159--4199},
  year={2025},
  publisher={European Mathematical Society}
}

@article{BrascoDePhilippisVelichkov,
  title={Faber--Krahn inequalities in sharp quantitative form},
  author={Brasco, Lorenzo and De Philippis, Guido and Velichkov, Bozhidar},
  journal={Duke Math. J.},
  volume={164},
  number={9},
  pages={1777--1831},
  year={2015},
  publisher={Duke University Press}
}

@article{Serrin,
  title={A symmetry problem in potential theory},
  author={Serrin, James},
  journal={Arch. Ration. Mech. Anal.},
  volume={43},
  number={4},
  pages={304--318},
  year={1971},
  publisher={Springer}
}

@article{GiaquintaGiusti,
  title = {On the regularity of the minima of variational integrals},
  volume = {148},
  journal = {Acta Math.},
  author = {Giaquinta,  Mariano and Giusti,  Enrico},
  year = {1982},
  pages = {31–46}
}

@misc{BDVV26,
      title={On a Sobolev critical problem for the superposition of a local and nonlocal operator with the ``wrong sign''}, 
      author={Stefano Biagi and Serena Dipierro and Enrico Valdinoci and Eugenio Vecchi},
      year={preprint, 2026},
      eprint={2601.07521},
}

@article{ACO09,
  title={Uniform energy distribution for an isoperimetric problem with long-range interactions},
  author={Alberti, Giovanni and Choksi, Rustum and Otto, Felix},
  journal={J. Amer. Math. Soc.},
  volume={22},
  number={2},
  pages={569--605},
  year={2009},
  publisher={American Mathematical Society}
}

@misc{DR26,
  author        = {Daneri, Sara and Runa, Eris},
  title         = {A rigorous approach to pattern formation for isotropic isoperimetric problems with competing nonlocal interactions},
  year          = {preprint, 2024},
}

@book{henrotpierre2018,
  title={Shape variation and optimization: A geometrical analysis},
  author={Henrot, Antoine and Pierre, Michel},
  volume={28},
  series={EMS Tracts in Mathematics},
  year={2018},
  publisher={European Mathematical Society (EMS)},
  address={Z{\"u}rich}
}

@book{Maggibook,
  title={Sets of Finite Perimeter and Geometric Variational Problems: An Introduction to Geometric Measure Theory},
  author={Maggi, Francesco},
  volume={135},
  series={Cambridge Studies in Advanced Mathematics},
  year={2012},
  publisher={Cambridge University Press},
  address={Cambridge}
}

@book{giaqmart,
  title={An Introduction to the Regularity Theory for Elliptic Systems, Harmonic Maps and Minimal Graphs},
  author={Giaquinta, Mariano and Martinazzi, Luca},
  series={Appunti. Scuola Normale Superiore di Pisa (Nuova Serie)},
  volume={11},
  year={2012},
  publisher={Edizioni della Normale},
  address={Pisa}
}

\end{document}